\documentclass[reqno,12pt]{amsart}
\usepackage{amsmath,amsthm,amssymb,amsfonts,amscd}
\usepackage{epsfig}
\usepackage{color}
\usepackage[all]{xy}
\usepackage{tikz-cd} 
\usepackage[hidelinks]{hyperref}
\usepackage{mathtools}

\theoremstyle{plain}   
\newtheorem{theorem}{Theorem}

\newtheorem{lemma}[theorem]{Lemma}
\newtheorem{corollary}[theorem]{Corollary}
\newtheorem{proposition}[theorem]{Proposition}
\newtheorem*{theorem*}{Theorem}
\newtheorem*{conjecture*}{Conjecture}
\newtheorem*{corollary*}{Corollary}

\theoremstyle{definition}
\newtheorem{remark}[theorem]{Remark}

\newtheorem{definition}[theorem]{Definition}

\makeatletter
\newsavebox{\@brx}
\newcommand{\llangle}[1][]{\savebox{\@brx}{\(\m@th{#1\langle}\)}%
  \mathopen{\copy\@brx\mkern2mu\kern-0.9\wd\@brx\usebox{\@brx}}}
\newcommand{\rrangle}[1][]{\savebox{\@brx}{\(\m@th{#1\rangle}\)}%
  \mathclose{\copy\@brx\mkern2mu\kern-0.9\wd\@brx\usebox{\@brx}}}
\makeatother

\usepackage{graphics, setspace}

\newcommand{\breakingcomma}{%
  \begingroup\lccode`~=`,
  \lowercase{\endgroup\expandafter\def\expandafter~\expandafter{~\penalty0 }}}
  
\usepackage{breqn}

\usepackage{amstext} 
\usepackage{array}
\newcolumntype{L}{>{$}l<{$}}

\allowdisplaybreaks

\newcommand{\CC}{{\mathbb{C}}}

\newcommand{\PP}{{\mathbb{P}}}
\newcommand{\QQ}{{\mathbb{Q}}}

\newcommand{\ZZ}{{\mathbb{Z}}}

\newcommand{\SL}{{\mathrm{SL}}}
\newcommand{\epi}{{\bf e}}

\newcommand{\F}[2]{{F_{#1}^{#2}}}

\newcommand{\Jac}{\mathrm{Jac}}
\newcommand{\Fix}{\mathrm{Fix}}

\newcommand{\Aut}{\mathrm{Aut}}
\newcommand{\id}{\mathrm{id}}

\newcommand{\pr}{{\mathrm{pr}}}
\newcommand{\bx}{{\bf x}}
\newcommand{\by}{{\bf y}}
\newcommand{\bt}{{\bf t}}
\newcommand{\bv}{{\bf v}}
\newcommand{\bs}{{\bf s}}
\newcommand{\bz}{{\bf z}}

\newcommand{\fAp}{{ f_{A^\prime} }}
\newcommand{\wdvv}{ {\mathrm{WDVV}} }

\def\A{{\mathcal A}}

\def\C{{\mathcal C}}

\def\F{{\mathcal F}}
\def\G{{\mathcal G}}

\def\O{{\mathcal O}}

\def\T{{\mathcal T}}
\def\X{{\mathcal X}}

\def\MF{{\mathrm{MF}}}

\def\Cl{{\mathrm{Cl}}}
\def\p{\partial }
\def\ns{{\nabla}\hspace{-1.4mm}\raisebox{0.3mm}{\text{\footnotesize{\bf /}}}}

\newcommand{\rmH}{{{\rm H}}}

\newcommand{\LL}{{\Upsilon}}

\newcommand{\ccHH}{{\mathsf{HH}}}

\begin{document}
\title{Mirror symmetry for a cusp polynomial Landau-Ginzburg orbifold}
\date{\today}
\author{Alexey Basalaev}
\address{A. Basalaev:\newline Faculty of Mathematics, National Research University Higher School of Economics, Usacheva str., 6, 119048 Moscow, Russian Federation, and \newline
Skolkovo Institute of Science and Technology, Nobelya str., 3, 121205 Moscow, Russian Federation}
\email{a.basalaev@skoltech.ru}
\author{Atsushi Takahashi}
\address{Department of Mathematics, Graduate School of Science, Osaka University, 
Toyonaka Osaka, 560-0043, Japan}
\email{takahashi@math.sci.osaka-u.ac.jp}

\begin{abstract}
For any triple of positive integers $A' = (a_1',a_2',a_3')$ and $c \in \CC^*$, cusp polynomial $\fAp = x_1^{a_1'}+x_2^{a_2'}+x_3^{a_3'}-c^{-1}x_1x_2x_3$ is known to be mirror to Geigle--Lenzing orbifold projective line $\PP^1_{a_1',a_2',a_3'}$. 
More precisely, with a suitable choice of a primitive form, Frobenius manifold of a cusp polynomial $\fAp$, turns out to be isomorphic to the Frobenius manifold of the Gromov--Witten theory of $\PP^1_{a_1',a_2',a_3'}$. 

In this paper we extend this mirror phenomenon to the equivariant case. Namely, for any $G$ --- a symmetry group of a cusp polynomial $\fAp$, we introduce the \textit{Frobenius manifold of a pair} $(\fAp,G)$ and show that it is isomorphic to the Frobenius manifold of the Gromov--Witten theory of Geigle--Lenzing weighted projective line $\PP^1_{A,\Lambda}$, indexed by another set $A$ and $\Lambda$, distinct points on $\CC\setminus\{0,1\}$.

For some special values of $A'$ with the special choice of $G$ it happens that $\PP^1_{A'} \cong \PP^1_{A,\Lambda}$. 
Combining our mirror symmetry isomorphism for the pair $(A,\Lambda)$,
together with the ``usual'' one for $A’$, we get certain identities of the coefficients
of the Frobenius potentials.
We show that these identities are equivalent to the identities between the Jacobi theta constants and Dedekind eta--function.
\end{abstract}
\maketitle

\setcounter{tocdepth}{1}
\tableofcontents

\section{Introduction}
Mirror symmetry conjectures the certain equivalence between hypersurface singularities and algebraic varieties. 
In the language of complex geometry, mirror symmetry conjectures an existence of an isomorphism between Frobenius manifolds, that are associated to both a holomorphic function, defining a singularity and an algebraic variety. Frobenius manifold of an algebraic variety $\X$ is given by Gromov-Witten theory of it, we denote it by $M_\X$ in what follows. 
Let $f = f(\bx)$ be a polynomial with only isolated critical points. Its Frobenius manifold is constructed by an unfolding of $f(\bx)$ after a special choice of a volume form $\zeta$, called K.Saito primitive form (cf. \cite{sa:1,st:2}). We will denote this Frobenius manifold by $M^{\zeta}_f$.

For any given triple of positive integers $A' = (a_1',a_2',a_3')$ and any fixed $c \in \CC^\ast$ consider the polynomial 
\[
    \fAp = \fAp(\bx) = x_1^{a_1'} + x_2^{a'_2} + x_3^{a'_3} - c^{-1}x_1x_2x_3.
\]
We call it \textit{cusp polynomial}.

It has an isolated critical point $x_1=x_2=x_3 = 0$. In particular, if $A'$ is such that $1/a_1' + 1/a_2' + 1/a_3' = 1$, the polynomial $\fAp$ defines a simple--elliptic singularity. Mirror to $\fAp$ is the so-called \textit{Geigle-Lenzing} orbifold projective line $\PP^1_{A'}$, that is the one-dimensional orbifold having at most three isotropic points of orders $a_1',a_2',a_3'$.

It was proved in \cite{st:3,mr,ms} that for $\fAp$ defining a simple--elliptic singularity there is is a choice of a primitive form $\zeta = \zeta^\infty$, so that the Frobenius manifold $M_{\fAp}^{\zeta^\infty}$ is isomorphic to the Frobenius manifold $M_{\PP^1_{A'}}$. Both Frobenius manifolds are of rank $\mu_{A'} = 2 + \sum_{i=1}^3 (a_i'-1)$. For the other choices of $A'$ the mirror symmetry conjecture was proved in \cite{st:1,ist:1}.

\subsection*{Equivariant approach}
This was an idea of the physicists that in mirror symmetry a polynomial with isolated critical points should always be assumed together with some symmetry group. 
From this point of view all the results mentioned above should be understood as being obtained for the trivial symmetry group $G = \{\id\}$.

For every fixed $A'$ consider $\fAp$ together with $G \subseteq G_\fAp$, where
\[
    G_\fAp := \left\lbrace g = \mathrm{diag}(g_1,g_2,g_3) \in \mathrm{GL}(3,\CC) \ | \ \fAp(\bx) = \fAp (g \cdot \bx)\right\rbrace. 
\]
The pair $(\fAp,G)$ is then called \textit{Landau-Ginzburg orbifold} (cf. \cite{IV}). In the equivariant approach, mirror symmetry conjectures that the pair $(\fAp,G)$ is mirror to some algebraic variety $\X$, depending effectively on both $\fAp$ and $G$. 

It was proposed by \cite{et:1} that mirror to $(\fAp,G)$ should be the orbifold $\PP^1_{A,\Lambda}$ again, but for \textit{the other set} $A = (a_1,\dots,a_r)$ defined beneath and $\Lambda$ --- the set of distinct points on $\CC\setminus\{0,1\}$. For $i=1,2,3$ let $K_i$ stand for the maximal subgroup of $G$, fixing $i$--th coordinate $x_i$. Define
\[
    A := \left( \frac{a_1'}{|G/K_1|} \ast |K_1| ,\frac{a_2'}{|G/K_2|} \ast |K_2|,\frac{a_3'}{|G/K_3|} \ast |K_3|\right),
\]
where $b \ast |K_i|$ means that $b$ is repeated $|K_i|$ times in the set. We also omit all appearences of $1$ in the set $A$.

Gromov--Witten theory of an orbifold $\PP^1_{A,\Lambda}$ is well--defined. Its Frobenius manifold $M_{\PP^1_{A,\Lambda}}$ was studied in \cite{shi:1}, it is of the rank $\mu_A = 2 + \sum_{i=1}^r (a_i-1)$. However there is no definition of a Frobenius manifold of a Landau--Ginzburg orbifold $(\fAp,G)$ and no ``equivariant'' version of a primitive form of K.Saito.


\subsection*{In this paper}
We address the problem of construction of a Frobenius manifold of a Landau--Ginzburg orbifold $(\fAp,G)$ with a primitive form $\zeta=\zeta^\infty$, such that mirror symmetry conjecture holds true.

The first step that needs to be completed on this way is the construction of \textit{orbifold version} of Jacobian algebra of $(\fAp,G)$. There are several way to consider it (cf. \cite{K,BTW16,BTW17,S20}). In this paper for such an algebra we consider Hochschild cohomology ring $\ccHH^*(\mathrm{MF}_G(\fAp))$ of the category of $G$--equivariant matrix factorizations of $\fAp$. We compute it explicilty employing the technique of \cite{S20}.

Next we define the Frobenius manifold $M_{(\fAp,G)}^{\zeta^\infty}$ of a Landau--Ginzburg orbifold $(\fAp,G)$ axiomatically. It also depends on the primitive form 
$\zeta^\infty$ of $\fAp$. In particular, we have for a trivial group $G$ that $M_{(\fAp,\{\id\})}^{\zeta^\infty} \cong M_{\fAp}^{\zeta^\infty}$. 
Using this system of axioms we show that the following mirror theorem holds.

\begin{theorem*}[Theorem~\ref{theorem: main} in the text]
  Suppose $M_{(f_{A'},G)}^{\zeta^\infty}$ satisfies all axioms of Frobenius manifold of the pair $(f_{A'},G)$.  
  Then there is Frobenius manifolds isomorphism:
  \begin{equation}
    M_{\PP^1_{A,\Lambda}} \cong M_{(f_{A'},G)}^{\zeta^\infty}.
  \end{equation}
  In particular the potential $\F_{(\fAp,G)}^\infty$ of $M_{(f_{A'},G)}^{\zeta^\infty}$ coincides with the genus $0$ potential of orbifold Gromov--Witten theory of $\PP^1_{A,\Lambda}$ after the certain choice of coordinates.
\end{theorem*}
\begin{corollary*}
    Frobenius manifold $M_{(f_{A'},G)}^{\zeta^\infty}$ is uniquely determined by its axioms.
\end{corollary*}

There is a related work on a construction of Frobenius manifold associated to the pair $(f,G)$ of an invertible polynomial $f$ (with a marginal deformation) and a subgroup $G$ of $\SL_N(\CC)\cap G_f$ (see \cite{Tu}).

However up to our knowledge this is the first mirror symmetry theorem for the Frobenius manifolds of Landau-Ginzburg orbifolds. In Section~\ref{section: examples} we provide plenty of examples of potentials $\F_{(\fAp,G)}^\infty$. It is important to note that we provide the mirror isomorphism of the theorem above explicilty.

\subsection*{\bf Simple--elliptic singularities and identities in the ring of quasimodular forms}
The definition of the set $A$ above assures that for any group $G$ whenever $A'$ is such that $\frac{1}{a_1'}+\frac{1}{a_2'}+\frac{1}{a_3'} = 1$, the set $A$ is either of length $3$ satisfying $\frac{1}{a_1}+\frac{1}{a_2}+\frac{1}{a_3} = 1$, or $A = (2,2,2,2)$. The respective Geigle-Lenzing orbifold projective lines are called \textit{elliptic orbifolds}, these are $\PP^1_{3,3,3}$, $\PP^1_{4,4,2}$, $\PP^1_{6,3,2}$ and $\PP^1_{2,2,2,2, \Lambda}$, $\Lambda\in \CC\setminus\{0,1\}$. As we have mentioned above, the first three orbifolds are known to be mirror to $(\fAp,\{\id\})$, latter being assumed together with the primitive form $\zeta^\infty$ (see also \cite{B16,B14} concerning the fourth one). 


Denote by $\widetilde E_6$ and $\widetilde E_7$ the singularities defined by $\fAp$ with $A'=(3,3,3)$ and $A' = (4,4,2)$ respectively. 
Let $M_{(\widetilde E_6,G)}^\infty$ and $M_{(\widetilde E_7,G)}^\infty$ stand for the corresponding Frobenius manifolds $M_{(\fAp,G)}^{\zeta^\infty}$.
Combining the mirror symmetry isomorphism with the trivial group $G = \{\id\}$ (on the right hand side) with the mirror symmetry isomorphism of our theorem (on the left hand side) we have the isomorphisms of Frobenius manifolds
\begin{align*}
    & M_{(\widetilde E_6,K_1)}^\infty 
    \quad \cong \quad
    M_{\PP^1_{3,3,3}} 
    \quad \cong \quad
    M_{(\widetilde E_6,\{\id\})}^\infty,
    \\
    & M_{(\widetilde E_7,K_2)}^\infty 
    \quad \cong \quad
    M_{\PP^1_{4,4,2}} 
    \quad \cong \quad
    M_{(\widetilde E_7,\{\id\})}^\infty.
\end{align*}

We show in Section~\ref{section: modular forms relations} that these isomorphisms are non--trivial. Namely, the coincidence of the corresponding Frobenius manifold potentials is equivalent to the certain non-trivial identities in the ring of quasimodular forms.

\subsection*{\bf Organization of the paper}
In Section~\ref{section: simple objects} we fix notation and introduce main objects. Section~\ref{section: frobenius manifolds} considers the Frobenius manifolds of a cusp polynomial and Gromov--Witten theory. Mirror symmetry results with the trivial symmetry group are recalled in Section~\ref{sec:mirror symmetry unorbifolded}. In Section~\ref{sec: HH} we compute Hochschild cohomology of the category of $G$--equivariant matrix factorizations of $\fAp(\bx)$. Frobenius manifold $M_{(\fAp,G)}^{\zeta^\infty}$ is introduced in Section~\ref{sec:symmetry group of affine cusp}. This is also the section where main theorem of this paper is introduced. We prove this theorem in Section~\ref{section: proof}. Section~\ref{section: examples} is devoted to the examples and Section~\ref{section: modular forms relations} to the identities between the quasimodular forms.

\subsection*{\bf Acknowledgements}
The first named author is supported by RSF Grant No. 19-71-00086.
The second named author is supported by JSPS KAKENHI Grant Number JP16H06337.

\section{Cusp polynomial and Geigle--Lenzing orbifold projective line}\label{section: simple objects}

\subsection{Cusp polynomial}\label{section: cusp polynomial}
Let $A'=(a'_1,a'_2,a'_3)$ be a triple of positive integers. We can associate to $A'$ the following polynomial 
\begin{equation}
x_1^{a'_1}+x_2^{a'_2}+x_3^{a'_3}-c^{-1}\cdot x_1 x_2 x_3,\quad c\in\CC\setminus\{0\},
\end{equation}
that we call {\it cusp polynomial}.
Consider its {\em universal unfolding}
\begin{equation}
F_{A'}({\bf x};{\bf s},s_{\mu_{A'}}):=x_1^{a'_1}+x_2^{a'_2}+x_3^{a'_3}-(s_{\mu_{A'}})^{-1}\cdot x_1 x_2 x_3
+s_1\cdot 1+\sum_{i=1}^{3} \sum_{j=1}^{a_i'-1}s_{i,j}\cdot x_i^j.
\end{equation}
For the later use, put $f_{A'}({\bf x};s_{\mu_{A'}}):=F_{A'}({\bf x};{\bf 0},s_{\mu_{A'}})$. Namely, we may identity $c$ above with the unfolding paramater $s_{\mu_{A'}}$.
We often regard $f_{A'}({\bf x};s_{\mu_{A'}})$ as a holomorphic map $\X^s\longrightarrow M^s$ 
for some suitable neighborhood of the origin $\X^s \subset \CC^{4}$ and $M^s:=\{s_{\mu_{A'}}\in \CC\backslash\{0\}~|~|\! |s_{\mu_{A'}}|\! |<\epsilon\}$.

In what follows we also use the notation: 
\begin{equation}
\mu_{A'}:=2+\sum_{i=1}^3 \left(a'_i-1\right),\quad \chi_{A'}:=2+\sum_{i=1}^3 \left(\frac{1}{a'_i}-1\right).
\end{equation}
For $\chi_{A'} \le 0$ the function $F_{A^\prime}$ is a miniversal unfolding considering $s_1$,$s_{i,j}$ and $s^{-1}_{\mu_{A^\prime}}$ as the unfolding parameters. 
Note that we have exactly $\mu_{A'}$ parameters in the unfolding $F_{A'}$. This number will become later the rank of the Frobenius manifold of cups polynomial. In what follows let $\bs$ vary in $\mathcal{S} := \CC^{\mu_{A'}-1} \times M^s$.

\begin{remark}\label{remark: affine cusp singularities}
If $\chi_{A'} < 0$, the point $0 \in \CC^3$ is not the only isolated critical point of $\fAp$. For the later purposes we do not need these additional critical points. Therefore  for any fixed $q$, we should consider $\fAp$ on a small neighborhood of $0 \in \CC^3$, not containing any other critical points.
\end{remark}

Associated to $\fAp$, for every fixed $c$, we consider the $\CC$--algebra 
\[
    \Jac(\fAp) := \O_{\CC^3,0}/ (\p_{x_1}\fAp,\p_{x_2}\fAp,\p_{x_3}\fAp) 
\]
with the following basis
\begin{equation}\label{eq: jac basis}
    e_1' := [1], \quad e_{\mu_{A'}}' := [x_1x_2x_3], \quad e_{i,j}' := [x_i^j], \quad 1\le i\le 3, 1 \le j \le a_i'-1.
\end{equation}
The vector $e_1'$ is the unit and the product structure of $\Jac(\fAp)$ given by
\begin{align}
    & e_{i_1,j_1}' \circ e_{i_2,j_2}' = 
    \begin{cases}
        \delta_{i_1,i_2} e_{i_1,j_1+j_2}' \quad & j_1 + j_2 < a_{i_1}',
        \\
        \dfrac{1}{c \cdot a_{i_1}'} e_{\mu_{A'}}' \quad & j_1+j_2 = a_{i_1}';
    \end{cases}
    \\
    & e_{i_1,j_1}' \circ e_{\mu_{A'}}' = 0.
\end{align}

The coordinates $s_\bullet$ are \textit{dual} to the basis we fix in the following sense.
\begin{align*}
    e_1' = \left[ \frac{\p F_{A'}}{\p s_1} \right], \ e_{\mu_{A'}}' = s_{\mu_{A'}}^2 \cdot \left[ \frac{\p F_{A'}}{\p s_{\mu_{A'}}} \right],
    \quad e_{i,j}' := \left[ \frac{\p F_{A'}}{\p s_{i,j}} \right]
\end{align*}

The algebra $\Jac(\fAp)$ with the product defined can be endowed with the pairing making it a Frobenius algebra. We will comment on this later because this pairing is only fixed after the certain addition choice is made - choice of a primitive form.

Assuming $c$ as a parameter, we can consider $\Jac(\fAp)$ as a $\CC(c)$--module. It has an \textit{extension} $\overline{\Jac}(\fAp)$, that is a $\CC[c]$--module, spanned by $[1]$, $[c^{-1}x_1x_2x_3]$ and $[x_i^j]$. By using the explicit form of the Jacobian ideal we have
\[
    \overline{\Jac}(\fAp) \mid_{c=0} \ \cong 
\CC[x_1,x_2,x_3]\left/\left(x_1x_2, x_2x_3, x_1x_3, \ a_1'x_1^{a_1'} - a_2'x_2^{a_2'}, a_2'x_2^{a_2'} - a_3'x_3^{a_3'} \right)\right. .
\]

\subsection{Orbifold projective line }\label{section: GL orbifold}
We introduce the variety $\PP^1_{A,\Lambda}$ to be the certain orbifold projective line of Geigle and Lenzing (see \cite{gl:1}).

For a natural $r \ge 3$ and for a $r$--tuple of positive integer numbers ${A = (a_1,\dots, a_r)}$ let $\Lambda=(\lambda_1,\ldots,\lambda_r)$ be a multiplet of pairwise distinct elements of $\PP^1$ normalized such that $\lambda_1=\infty$, $\lambda_2=0$ and $\lambda_3=1$. 

\begin{definition}\label{orb proj line} \
  \begin{enumerate}
    \item
      Define a ring $R_{A,\Lambda}$ by 
      \begin{subequations}
	\begin{equation*} 
	  R_{A,\Lambda} := \CC[X_1,\dots,X_r]\left/I_{\Lambda}\right.,
	\end{equation*}
	where  $I_\Lambda$ is an ideal generated by the $r-2$ homogeneous polynomials
	\begin{equation*}
	  X_i^{a_i} - X_2^{a_2} + \lambda_i X_1^{a_1},\quad i=3,\dots, r.
	\end{equation*}
      \end{subequations}
    \item
      Denote by $L_{A}$ an abelian group, generated by $r$--letters $\vec{X_i}$, $i=1,\dots ,r$, by
      \begin{subequations}
	\begin{equation*}
	  L_{A} := \bigoplus_{i=1}^r\ZZ\vec{X}_i \left / M_{A}\right.,
	\end{equation*}
	for being $M_{A}$ the subgroup generated by the elements
	\begin{equation*}
	  a_i\vec{X}_i - a_j\vec{X}_j, \quad 1\le i<j\le r.
	\end{equation*}
      \end{subequations}
    \item
      The \textit{orbifold projective line} of type $(A,\Lambda)$ in the quotient stack $\PP^1_{A,\Lambda}$ defined by:
      \begin{equation*}
	\PP^1_{A,\Lambda}:=\left[\left({\rm Spec}(R_{A,\Lambda})\backslash\{0\}\right)/{\rm Spec}({\CC L_{A}})\right],
      \end{equation*}
  \end{enumerate}
\end{definition}
An orbifold projective line of type $(A,\Lambda)$ is a Deligne--Mumford stack whose coarse moduli space is 
a smooth projective line $\PP^1$. Then the numbers $\lambda_1, \dots, \lambda_n$ are the coordinates of the projection on $\PP^1$ of the points having a non--trivial stabilizer.
In what follows we skip the letter $\Lambda$ in the notaion of $\PP^1_{A,\Lambda}$ when $r=3$.

Following \cite{cr:1} we associate to $\PP^1_{A,\Lambda}$ the orbifold cohomology ring $H^*_{orb}\left( \PP^1_{A, \Lambda}, \QQ \right)$. It is an associative commutative algera with the basis $\Delta_1, \Delta_{\mu}$, $\Delta_{i,j}$ with $1 \le i \le r$, $1 \le j \le a_i-1$ such that:
\begin{equation}\label{eq: qcoh basis}
H^0_{orb}(\PP^1_{A,\Lambda}) \cong \QQ \Delta_1,\ H^2_{orb}(\PP^1_{A,\Lambda}) \cong \QQ \Delta_\mu, \ \Delta_{i,j} \in H^{2\frac{j}{a_i}}_{orb}(\PP^1_{A,\Lambda}),
\end{equation}
and the pairing $\eta$ having only the following non--zero values:
$$
\eta \left( \Delta_1, \Delta_\mu \right) = 1, \ \eta \left( \Delta_{i,j}, \Delta_{i,a_i-j} \right) = 1/a_i. 
$$

Compared to the case of $\Jac(\fAp)$ above, here we can not introduce the product strucure before giving the certain additional data --- the three-point correlators of Gromov--Witten theory of $\PP^1_{A,\Lambda}$.

\subsection{Symmetry group}

Let $G$ be a finite abelian subgroup of ${\rm SL}(3,\CC)$ acting diagonally on $\CC^3$ such that $f_{A'}({\bf x})$ is invariant under its natural action. 
For each $g \in G$, denote by $N_g$ the dimension of the fixed locus which is a linear subspace of $\CC^3$ and by $d_g := N - N_g$.
Denote $\epi[x] = \exp(2\pi \sqrt{-1} \cdot x)$. Each element $g\in G$ has a unique expression of the form
\begin{equation}
g={\rm diag}\left (\epi\left[\frac{k_1}{r}\right], \epi\left[\frac{k_2}{r}\right],\epi\left[\frac{k_3}{r}\right] \right) \quad \mbox{with } 0 \leq k_i < r,
\end{equation}
where $r$ is the order of $g$. 
The {\em age} of $g$, which is introduced in \cite{IR}, is defined as the rational number 
\begin{equation}
{\rm age}(g) := \frac{1}{r}\sum_{i=1}^3 k_i. 
\end{equation}
Since we assume that $G\subset {\rm SL}(3,\CC)$, this number is an integer. 
Define $j_G$ to be the number of elements  $g \in G$ such that ${\rm age}(g)=1$ and $N_g=0$. 
If $g$ is an element of age $1$ 
, then $g^{-1}
$ is an element of age $2$. Therefore, the number $j_G$ is also the number of elements $g \in G$ such that ${\rm age}(g)=2$.

For $i=1,2,3$, let $K_i$ be the maximal subgroup of $G$ fixing the $i$-th coordinate $x_i$, whose 
order $|K_i|$ is denoted by $n_i$.
\begin{proposition}[{\cite[Corollary~2]{et:1}}]\label{prop:G}
We have
\begin{equation}
\left|G\right|=1+2j_G+\sum_{i=1}^3\left(n_i-1\right),
\end{equation}
where $j_G$ is the number of elements  $g \in G$ such that ${\rm age}(g)=1$ and $N_g=0$.
\end{proposition}

For any $g \in G$ let $\Fix(g) \subseteq \CC^3$ denote the fixed locus of $g$ and $I_g^c$ be a subset of $\{1,2,3\}$, s.t. $g(x_k) \neq x_k$ for $k \in I_g^c$.
Let $f^g$ be the following polynomial
\[
    f^g := \fAp \mid_{x_k =0, \ k \in I_g^c} \ = \fAp \mid_{\Fix(g)}.
\]
We have $f^g = x_i^{a'_i}$ for $g \in K_i$.

For $i=1,2,3$, set
\begin{equation}
a_i:=\frac{a'_i}{\left|G/K_i\right|},
\end{equation}
and define a tuple of positive integers $A=(a_1, \ldots , a_r)$ by
\begin{equation}
(a_1, \ldots, a_r) = \left( \frac{a'_i}{|G/K_i|} \ast |K_i|, i=1,2,3 \right) ,
\end{equation}
where $u \ast v$ denotes $v$ copies of the integer $u$ and we omit numbers equal to one on the right-hand side. 
Set 
\begin{equation}
\mu_{A}:=2+\sum_{i=1}^r \left(a_i-1\right),\quad \chi_{A}:=2+\sum_{i=1}^r \left(\frac{1}{a_i}-1\right).
\end{equation}  

\section{Frobenius manifolds of a cusp polynomial and Gromov-Witten theory}\label{section: frobenius manifolds}

Frobenius manifolds were introduced by B.Dubrovin (cf. \cite{du:2}). Important examples of Frobenius manifolds originate from singularity theory and Gromov--Witten theory.

\begin{definition}
Let $M=(M,\O_{M})$ be a connected complex manifold of dimension $\mu$
whose holomorphic tangent sheaf and cotangent sheaf 
are denoted by $\T_{M}$ and $\Omega_M^1$ respectively
and let $d$ be a complex number.
A {\it Frobenius structure of rank $\mu$ and conformal dimension $d$ on M} is a tuple $(\eta, \circ , e,E)$, where $\eta$ is a non--degenerate $\ZZ / 2\ZZ$--graded $\O_{M}$--symmetric bilinear form on $\T_{M}$, $\circ $ is an $\O_{M}$--bilinear product on $\T_{M}$, defining an associative and commutative $\O_{M}$--algebra structure with a unit $e$, and $E$ is a holomorphic vector field on $M$, called the Euler vector field, which are subject to the following axioms:
\begin{enumerate}
\item The product $\circ$ is self--adjoint with respect to $\eta$: that is,
\begin{equation*}
\eta(\delta\circ\delta',\delta'')=\eta(\delta,\delta'\circ\delta''),\quad
\delta,\delta',\delta''\in\T_M. 
\end{equation*} 
\item The {\rm Levi}--{\rm Civita} connection $\ns:\T_M\otimes_{\O_M}\T_M\to\T_M$ 
with respect to $\eta$ is flat. That is, 
\begin{equation*}
[\ns_\delta,\ns_{\delta'}]=\ns_{[\delta,\delta']},\quad \delta,\delta'\in\T_M.
\end{equation*}
\item The tensor $C:\T_M\otimes_{\O_M}\T_M\to \T_M$  defined by 
$C_\delta\delta':=\delta\circ\delta'$, $(\delta,\delta'\in\T_M)$ is flat: that is,
\begin{equation*}
\ns C=0.
\end{equation*} 
\item The unit element $e$ of the $\circ $-algebra is a 
$\ns$-flat holomorphic vector field: that is,
\begin{equation*}
\ns e=0.
\end{equation*} 
\item The metric $\eta$ and the product $\circ$ are homogeneous of degree 
$2-d$ ($d\in\CC$) and $1$ respectively with respect to the Lie derivative 
$Lie_{E}$ of the {\rm Euler} vector field $E$: that is,
\begin{equation*}
Lie_E(\eta)=(2-d)\eta,\quad Lie_E(\circ)=\circ.
\end{equation*}
\end{enumerate}
A manifold $M$ equipped with a Frobenius structure $(\eta, \circ , e,E)$ is called a {\it Frobenius manifold}.
\end{definition}

The structure of a Frobenius manifold can be locally described by the analytic function $\F$, called \textit{potential}. Namely, let $n = \dim M$ and $t_1,\dots,t_n$ be flat coodinates of the Levi-Civita connection above. At a point $p \in M$ consider let $T_pM = \CC \langle \frac{\p}{\p t_1} ,\dots,\frac{\p}{\p t_n}\rangle$. Assume in addition that $t_1$ is such that $e = \frac{\p}{\p t_1}$, $\eta_{ij} = \eta(\frac{\p}{\p t_i},\frac{\p}{\p t_j})$ --the components of $\eta$ in the basis fixed and $\eta^{ij}$ being components of $\eta^{-1}$. Then there is a function $\F=\F(t_1,\dots,t_n)$, such that
\begin{align*}
    & \frac{\p}{\p t_i} \circ \frac{\p}{\p t_j} = \sum_{k,l=1}^n \frac{\p^3 \F}{\p t_i \p t_j \p t_l} \eta^{lk} \frac{\p}{\p t_k},
    \\
    & E \cdot \F = (3 - d) \F + \text{terms, quadratic in $t_\bullet$}.
\end{align*}
Locally the potential $\F$ fully encodes the data of a Frobenius manifold $M$.

\subsection{Orbifold Gromov-Witten theory}
In \cite{cr:1} the authors gave the treatment of GW-theory for an orbifold $\X$. Let $I\X$ be the inertia orbifold of $\X$. Fixing $\beta \in H_2(I\X, \mathbb Z)$ the authors define the moduli space $\overline {\mathcal M}_{g,n}(\X, \beta)$ of degree $\beta$ stable orbifold maps from the genus $g$ curve with $n$ marked points to $I\X$. 
Together with the suitable fundamental cycle $[ \overline {\mathcal M}_{g,n}(\X, \beta) ] ^{vir}$ one can introduce the correlators. Define ${ev_i: \overline {\mathcal M}_{g,n}(\X, \beta) \rightarrow I\X}$ -- the map sending the stable orbifold map with $n$ markings to its value at the $i$-th marked point.

Let $\gamma_i \in H^*_{orb} (I\X, \mathbb Q)$ -- the elements of the Chen-Ruan orbifold cohomology ring. The Gromov-Witten theory correlators are defined by:
$$
    \langle \gamma_1 , \dots , \gamma_n \rangle_{g,n,\beta}^\X := \int_{ [ \overline {\mathcal M}_{g,n}(\X, \beta) ] ^{vir} } ev_1^* \gamma_1 \wedge \dots \wedge ev_n^* \gamma_n.
$$
It is convenient to assemble the numbers obtained into a generating function called genus $g$ potential of the orbifold GW theory. Let 
\[
    \bt := t_1 \Delta_1 + t_{\mu} \Delta_{\mu} +  \sum_{i=1}^r\sum_{j=1}^{a_r-1} \Delta_{i,j} t_{i,j}
\] 
for the formal parameters $t_\bullet$ and the basis of $H^*_{orb}(I\X, \mathbb Q)$ as in \eqref{eq: qcoh basis}. This definition fixes the certain connection between $t_\bullet$ and $\Delta_\bullet$. Namely, we have $\dfrac{\p \bt}{\p t_\bullet} = \Delta_\bullet$. We will call the coordinate $t_\bullet$ \textit{dual} to $\Delta_\bullet$.

The genus $g$ potential is a formal power series in $t_\bullet$. It reads
$$
    \mathcal F^{\X}_g := \sum_{n, \beta} \frac{1}{n!} \langle \textbf t , \dots , \textbf t \rangle_{g,n,\beta}^\X.
$$
The most important for us will be the genus zero potential. Due to the geometrical properties of the moduli space of curves, it is a solution to the WDVV equation and defines a Frobenius manifold that we denote by $M_\X$.

We have the following first examples.
\begin{align*}
 \F_{\PP^1_{2,2,2}} &= \frac{q^4}{4} + \frac{q^2}{2} \left(t_{1,1}^2+t_{2,1}^2+t_{3,1}^2\right) + q t_{1,1} t_{2,1} t_{3,1} 
 \\
 & 
 -\frac{1}{96}\left( t_{1,1}^4 + t_{2,1}^4 + t_{3,1}^4\right) +\frac{1}{4} t_1 \left(t_{1,1}^2+t_{2,1}^2+t_{3,1}^2\right) + \frac{1}{2} t_{5} t_1^2,
\\
    \F_{\PP^1_{2,3,3}} &= \frac{q^{12}}{12} +\frac{1}{2} q^8 t_{2,2} t_{3,2} + q^6 \left(\frac{t_{2,2}^3}{6}+\frac{t_{3,2}^3}{6}+\frac{t_{1,1}^2}{2}\right) +q^5 t_{1,1} t_{2,2} t_{3,2}
    \\
    &\ +q^4 \left(\frac{5}{18} t_{3,2}^2 t_{2,2}^2+\frac{1}{6} t_{3,1} t_{2,2}^2+\frac{1}{6} t_{2,1} t_{3,2}^2+t_{2,1} t_{3,1}\right) 
    \\
    &\ +q^3 \left(\frac{t_{1,1}^3}{3}+\frac{1}{6} t_{2,2}^3 t_{1,1}+\frac{1}{6} t_{3,2}^3 t_{1,1}+t_{2,1} t_{2,2} t_{1,1}+t_{3,1} t_{3,2} t_{1,1}\right) 
    \\
    &\ +q^2 \left(\frac{1}{72} t_{3,2} t_{2,2}^4+\frac{1}{6} t_{2,1} t_{3,2} t_{2,2}^2+\frac{1}{72} t_{3,2}^4 t_{2,2} \right.
    \\
    &\quad \left.+\frac{1}{2} t_{3,1}^2 t_{2,2}+\frac{1}{6} t_{3,1} t_{3,2}^2 t_{2,2}+\frac{1}{2} t_{1,1}^2 t_{3,2} t_{2,2}+\frac{1}{2} t_{2,1}^2 t_{3,2}\right)
    \\
    &\ + q \left(\frac{1}{36} t_{1,1} t_{3,2}^2 t_{2,2}^2+\frac{1}{6} t_{1,1} t_{3,1} t_{2,2}^2+\frac{1}{6} t_{1,1} t_{2,1} t_{3,2}^2+t_{1,1} t_{2,1} t_{3,1}\right)
    \\
    &\ -\frac{t_{2,2}^6 + t_{3,2}^6}{19440} +\frac{1}{648} t_{2,1} t_{2,2}^4+\frac{1}{648} t_{3,1} t_{3,2}^4-\frac{t_{1,1}^4}{96} + \frac{t_{2,1}^3}{18}+\frac{t_{3,1}^3}{18}
    \\
    &\ -\frac{1}{36} t_{2,2}^2 t_{2,1}^2-\frac{1}{36} t_{3,1}^2 t_{3,2}^2
    +\frac{1}{4} t_1 t_{1,1}^2+\frac{1}{3} t_1 t_{2,1} t_{2,2}+\frac{1}{3} t_1 t_{3,1} t_{3,2}+\frac{1}{2} t_{6} t_1^2,
\end{align*}
where we use $q := \exp(t_{\mu})$.

These potentials were found by P.Rossi (cf. \cite[Example 3.2]{rossi}). However he had a missprint in $\F_{\PP^1_{2,2,2}}$ that we fix above.

\subsection{Singularity theory}
For any fixed $f \in \O_{\CC^N,0}$ defining an isolated singularity let $\Jac(f) := \O_{\CC^N,0}/ (\p_{x_1}f,\dots, \p_{x_N}f)$ and $\mu := \dim\Jac(f)$. Given an unfolding $F: \CC^n \times \CC^\mu \to \CC$ of $f$, introduce Frobenius manifold structure on the base space $\mathcal{S}$ of an unfolding. This is a $\mu$--dimensional open subspace of $\CC^\mu$. The product $\circ$ is induced from the critical sheaf 
\[
 \O_\C := \O_{\CC^n\times \mathcal{S}} / (\p_{x_1}F, \dots, \p_{x_n}F)
\]
by the projection $p: \CC^n\otimes \mathcal{S} \to \CC^n$. Namely, we have an isomorphism $\T_{\mathcal{S}} \to p_* \O_\C$ given by $\delta \mapsto [\delta \cdot F]$. The product $\circ$ on $\T_{\mathcal{S}}$ is then defined as a pullback of the natural product of $p_*\O_\C$ via this isomorphism. 

The pairing $\eta$ is defined via the Poincare residue pairing. However, this requires an additional choice of a volume form. 
The result of K.Saito assures that for a special choice of a volume form called a \textit{primitive form}, there is a flat connection $\ns$ that is metric with respect to $\eta$ (cf. \cite{sa:1}).
In what follows denote by $M_{f}^\zeta$ the Frobenius manifold structure defined by $f$ with a primitive form $\zeta$. 
The details of this construction can be found in \cite{sa:1,st:2} and \cite{st:1} for the case of cusp polynomials.

For any primitive form $\zeta$ and a non--zero constant $c \in \CC$, $\zeta' := c \zeta$ is again a primitive form. The two Frobenius manifolds $M_f^\zeta$ and $M_f^{\zeta'}$ have the same product structure, but the pairings are related by $\eta' = c^2 \cdot \eta$. In general not all primitive forms of a given singularity are connected by a simple rescaling. Different choices of a primitive form can give Frobenius manifolds that are not isomorphic (cf. \cite{ms}).

Choice of a primitive form seriously affects the potential of the Frobenius manifold $M_f^\zeta$. Recall that the potential is written in flat coodinates of $\eta$. The choice of a primitive form fixed the expression of the flat coodinates $\bt$ via the base space coordinates $\bs$.

\subsubsection{Cusp polynomials}
For $\fAp$ being a cusp polynomial, one considers the unfolding $F_{A'}$ as in Section~\ref{section: cusp polynomial}. 
The primitive forms for this unfolding were considered by K.Saito and Ishibashi--Shiraishi--Takahashi (see \cite{sa:1} and \cite{ist:1} for different values of $\chi_{A^\prime}$). In particular, there is a special choice of primitive form $\zeta = \zeta^\infty$, such that
\begin{align*}
  &\zeta^\infty = d^3\bx \ s_{\mu_{A^\prime}}^{-1} && \text{ for } \chi_{A^\prime} > 0,
   \\
   &\zeta^\infty = d^3\bx (s_{\mu_{A^\prime}}^{-1}  + O(\bs)), &&\text{ for }  \chi_{A^\prime} \le 0.
\end{align*}
These primitive forms are called \textit{primitive forms at infinity}.
In what follows we denote by $M_\fAp^\infty$ the Frobenius manifold of a cusp polynomial with the primitive form $\zeta^\infty$. Let $\F_{\fAp}^\infty$ be its potential, written in the flat coordinates $t_\bullet$ dual to the basis \eqref{eq: jac basis}.

The connection between these flat coordinates $t_\bullet$ and coordinates $s_\bullet$ of the unfolding is given by the functions $t_\bullet = t_\bullet(\bs)$, satisfying 
\[
    s_{\mu_{A'}} = \exp(t_{\mu_{A'}}), \ \frac{\p t_\alpha}{\p s_\beta}\mid_{\bs = 0, \ s_{\mu_{A'}} = 0} = \delta_{\alpha,\beta}, \quad t_{\alpha}\mid_{\bs = 0, \ s_{\mu_{A'}} = 0} = 0,
\]
for all indices $\alpha,\beta$. We have
\begin{align*}
& t_{1} \text{ is the flat coordinate dual to } [1] \text{ at the limit } \bs=s_{\mu_{A'}}=0,
\\
& t_{i,j} \text{ is the flat coordinate dual to } [x_{i}]^j \text{ at the limit } \bs=s_{\mu_{A'}}=0,
\\
& t_{\mu_{A'}}  \text{ is the flat coordinate dual to } (s_{\mu_{A'}})^{-1}[x_1x_2x_3] \text{ at the limit } \bs=s_{\mu_{A'}}=0. 
\end{align*}
Recall that the classes on the right hand side are exactly the classes spanning $\overline{\Jac}(\fAp)$.

In the flat coordinates the only non--zero values of $\eta(\cdot,\cdot)$ are
\[
    \eta\left(\frac{\p}{\p t_1},\frac{\p}{\p t_{\mu_{A'}}}\right) = 1,
    \quad
    \eta\left(\frac{\p}{\p t_{i,j}},\frac{\p}{\p t_{i,a_i'-j}}\right) = \frac{1}{a_i'}, \ 1 \le i \le 3, \ 1 \le j\le a_i-1.
\]
The Euler field reads
\[
    E = t_1 \frac{\p}{\p t_1} + \sum_{i=1}^3\sum_{j=1}^{a_i'-1} \frac{a_i'-j}{a_i'} t_{i,j} \frac{\p}{\p t_{i,j}} + \chi_{A'} \frac{\p}{\p t_{\mu_{A'}}}.
\]

The details of the construction of these Frobenius manifolds can be found in \cite{ist:2}.

\subsection{$G$--invariants of $M_{\fAp}^\infty$}
For a given cusp polynomial $\fAp(\bs)$ and $G \subset G_{\fAp}$, we can consider $\A_{\fAp}^G := \left( \Jac(\fAp) \right)^G$. Namely, the $G$--invariant subspace of $\Jac(\fAp)$. In particular, it has the basis (compare it to \eqref{eq: jac basis})
\begin{equation}
    e_1 := [1], \quad e_{\mu_{A}} := [x_1x_2x_3], \quad e_{i,j} := [x_i^{j\cdot n_i}], \quad 1\le i\le 3, 1 \le j \le a_i-1.
\end{equation}
Consider the coordinates $\bs$ of the unfolding $F(\bx,\bs)$. We call the coordinates $s_\bullet$, dual to these vectors, \textit{G--invariant}. Namely, these are $s_1$, $s_{\mu_{A'}}$ and $s_{i,j\cdot n_i}$ for $i,j$ as above. Set
\[
    \mathcal{S}^G := \mathcal{S} \mid_{s_{i,k} = 0, \ k \not\in n_i \ZZ}.
\]
Obviously, the $G$--invariant coordinates $s_\bullet$ are coordinates of $\mathcal{S}^G$.

Let $(\eta,\circ,e,E)$ be the Frobenius manifold structure of $M_\fAp^\infty$. It follows immediately by the definition that the product $\circ: \T_{\mathcal{S}} \otimes \T_{\mathcal{S}} \to \T_{\mathcal{S}}$ descends to the associative and commutative product $\widetilde\circ: \T_{\mathcal{S}^G} \otimes \T_{\mathcal{S}^G} \to \T_{\mathcal{S}^G}$.

Consider now the flat coordinates $t_\bullet = t_\bullet(\bs)$ above. We have
\begin{align}
    &t_{1} \mid_{\mathcal{S}^G} \ = s_{1} + \text{terms, at least quadratic in $\bs$},
    \\
    &t_{i,j \cdot n_i} \mid_{\mathcal{S}^G} \ = s_{i, j \cdot n_i} + \text{terms, at least quadratic in $\bs$},
    \\
    &t_{i,k} \mid_{\mathcal{S}^G} \ = \text{terms, at least quadratic in $\bs$}, \quad k \not\in n_i\ZZ.
\end{align}
The special coordinate $t_{\mu_{A'}}$ is not changed by restricting to $\mathcal{S}^G$.
It follows that $t_1$, $t_{\mu_{A'}}$ and $t_{i,j\cdot n_i}$, being restricted to $\mathcal{S}^G$, can serve as the coodinates too. Note that $\eta$, restricted to these coordinates is non--degenerate. 

We can consider the restriction of $(\eta,\circ,e,E)$ to the $G$--invariants. Denote
\[
    \left(M_\fAp^\infty\right)^G := M_\fAp^\infty \mid_{t_{i,k} = 0, \ k \not\in n_i\ZZ}.
\]
The vector field $e$ belongs to the tangent sheaf of $\left(M_\fAp^\infty\right)^G$. The vector field $E$ is projected to this tangent sheaf in a straightforward way. By the discussion above we have that $\widetilde\circ$ gives a suitable product and $\eta$ restricts too.

Summing all together we get that $(M_\fAp^\infty)^G$ is a Frobenius submanifold of $M_\fAp^\infty$. 
Let $\F_\fAp$ be the potential of $M_\fAp^\infty$. Then the potential $\F_\fAp^G$ of $(M_\fAp^\infty)^G$ is obtained from $\F_\fAp$ by setting $t_{i,k} = 0$, for all $1\le i \le 3$ and $k \not\in n_i\ZZ$.


\section{Mirror symmetry with a trivial symmetry group}\label{sec:mirror symmetry unorbifolded}
In this section we recall mirror symmetry results with the trivial symmetry group. 

\subsection{Mirror symmetry theorem}
The following mirror symmetry theorem should be considered as the aggregate result of several different papers: \cite{st:1}, \cite{st:3}, \cite{ist:2}, \cite{rossi}, \cite{mr}, \cite{ms}.

\begin{theorem}[{\cite[Theorem~4.1]{st:1}, \cite[Theorem~3.6]{st:3}, \cite[Corollary~4.5]{ist:2}}]\label{theorem: msUnorbifolded}
  Frobenius manifold 
  of $\fAp$ with the unfolding $F_{A^\prime}$ and the primitive form $\zeta^\infty$ is isomorphic to the Frobenius manifold of the Gromov--Witten theory of $\PP_{A^\prime}$. 
\end{theorem}

Writing the two Frobenius manifolds in the coordinates $t_\bullet$ we have used in Section~\ref{section: frobenius manifolds}, the mirror isomorphism $M_{\fAp}^\infty \to M_{\PP^1_{A'}}$ is given by
\[
    t_1 \mapsto t_1, \ t_{\mu_{A'}} \mapsto t_{\mu_{A'}}, \ t_{i,j} \mapsto t_{i,j}.
\]

The proof of this theorem requires a subtle analysis of the Gauss-Manin connection on the unfolding space $\mathcal{S}$ of $\fAp$ and also the certain uniqueness theorem. On the level of Frobenius manifolds such uniqueness theorem was first formulated in \cite{ist:1}. In what follows we use another uniqueness theorem by Y.Shiraishi, generalizing the latter one.

\begin{theorem}[{\cite[Theorem~3.1]{shi:1}}]\label{thm:satisfies ISTr}
Let $r$, $A$, $\mu_A$ and $\chi_A$ be as in Section~\ref{section: GL orbifold}.
There exists a unique Frobenius manifold $M$ of rank $\mu_A$ and conformal dimension one with flat coordinates 
$(t_1,t_{1,1},\dots ,t_{i,j},\dots ,t_{r,a_r-1},t_{\mu_A})$ satisfying the following conditions$:$
\begin{enumerate}
\item[(i)]
The unit vector field $e$ and the Euler vector field $E$ are given by
\[
e=\frac{\p}{\p t_1},\ E=t_1\frac{\p}{\p t_1}+\sum_{i=1}^{r}\sum_{j=1}^{a_i-1}\frac{a_i-j}{a_i}t_{i,j}\frac{\p}{\p t_{i,j}}
+\chi_A\frac{\p}{\p t_{\mu_A}}.
\]
\item[(ii)]
The non-degenerate symmetric bilinear form $\eta$ on $\T_M$ satisfies
\begin{align*}
&\ \eta\left(\frac{\p}{\p t_1}, \frac{\p}{\p t_{\mu_A}}\right)=
\eta\left(\frac{\p}{\p t_{\mu_A}}, \frac{\p}{\p t_1}\right)=1,\\ 
&\ \eta\left(\frac{\p}{\p t_{i_1,j_1}}, \frac{\p}{\p t_{i_2,j_2}}\right)=
\begin{cases}
\frac{1}{a_{i_1}}\quad i_1=i_2\text{ and }j_2=a_{i_1}-j_1,\\
0 \quad \text{otherwise}.
\end{cases}
\end{align*}
\item[(iii)]
The Frobenius potential $\F$ satisfies $E\F|_{t_{1}=0}=2\F|_{t_{1}=0}$,
\[
\left.\F\right|_{t_1=0}\in\CC\left[[t_{1,1}, \dots, t_{1,a_1-1}, 
\dots, t_{i,j},\dots, t_{r,1}, \dots, t_{r,a_r-1},e^{t_{\mu_A}}]\right].
\]
\item[(iv)] Assume the condition {\rm (iii)}. we have
\begin{equation*}
\F|_{t_1=e^{t_{\mu_A}}=0}=\sum_{i=1}^{r}\G^{(i)}, \quad \G^{(i)}\in \CC[[t_{i,1},\dots, t_{i,a_i-1}]],\ i=1,\dots,r.
\end{equation*}
\item[(v)] 
Assume the condition {\rm (iii)}. In the frame $\frac{\p}{\p t_1}, \frac{\p}{\p t_{1,1}},\dots, 
\frac{\p}{\p t_{r,a_r-1}},\frac{\p}{\p t_{\mu_A}}$ of $\T_M$,
the product $\circ$ can be extended to the limit $t_1=t_{1,1}=\dots=t_{r,a_r-1}=e^{t_{\mu_A}}=0$.
The $\CC$-algebra obtained in this limit is isomorphic to
\[
\CC[x_1,x_2,\dots, x_r]\left/\left(x_ix_j, \ a_ix_i^{a_i}-a_jx_j^{a_j}
\right)_{1\le i< j\le r}\right.,
\]
where $\p/\p t_{i,j}$ are mapped to
$x^{j}_i$ for $i=1,\dots,r, j=1,\dots, a_{i}-1$ and $\p/\p t_{\mu_A}$ is mapped to $a_{1}x_{1}^{a_1}$.
\item[(vi)] The term 
\[
\displaystyle\left(\prod_{i=1}^{r}t_{i,1}\right)e^{t_{\mu_A}}
\]
occurs with the coefficient $1$ in $\F$. 
\end{enumerate}
\end{theorem}
For the mirror symmetry purposes another theorem by Y.Shiraishi (see also \cite{ist:1} for $r=3$ case) is helpful. 

\begin{theorem}[{\cite[Theorem~5.5]{shi:1}}]\label{theorem: yuuki--GW}
  The conditions of Theorem~\ref{thm:satisfies ISTr} are satisfied by the genus zero potential of the Gromov--Witten theory of $\PP^1_{A, \Lambda}$ written in the basis~\eqref{eq: qcoh basis}.
\end{theorem}

The two theorems combined give that the Frobenius manifold $M_{\fAp}^\zeta$ is isomorphic to $M_{\PP^1_{A'}}$ if and only if there is a choice of $\zeta$ and the suitable coordinates $\bt$, such that $\F_{\fAp}^\zeta(\bt)$ satisfies all conditions of Theorem~\ref{thm:satisfies ISTr}.

In \cite{ist:1} the authors show that $\F_\fAp^\infty$ satisfies all conditions of Theorem~\ref{theorem: yuuki--GW} above with $r = 3$. In particular, condition~(v) is fulfilled by the algebra $\overline{\Jac}(f_{A^\prime})$. 
In the next section we develop an 'orbifold' version of it that will fulfill condtion~(v) for our mirror symmetry theorem.

\section{Hochschild cohomology of a cusp polynomial}\label{sec: HH}
In this section we compute $\ccHH^*(\MF_G(\fAp))$. 
Recalling that $\fAp$ has an additional complex parameter $c$, we consider this Hochschild cohomology and related rings beneath as $\CC(c)$--modules.  

To perform the computations we make use of technique developed by Shklyarov in~\cite{S20}. 

\subsection{Shklyarov's technique in Hochschild cohomology}\label{sec:Shk}
Shklyarov introduces a $\ZZ/2\ZZ$-graded $G$-twisted commutative algebra $\A^*(f,G)$ 
whose underlying  $\CC$-vector space is given by 
\begin{equation}
\A^*(f,G) = \bigoplus_{g \in G} \Jac(f^g) \xi_g,
\end{equation}
where $\xi_g$ is a generator (a formal letter) attached to each $g\in G$.
It is required that the group $G$ acts on $\A^*(f,G)$ via a coodinate-wise action on $\Jac(f^g)$ and a generator $\xi_g$ being transformed as
\begin{equation}\label{G-action2}
G\ni h=(h_1, \ldots, h_n): \quad \xi_g\mapsto \prod_{i\in I_g^c} h^{-1}_i\,\cdot  \xi_g.
\end{equation}
so that the product structure of $\A^*(f,G)$ is invariant under the $G$-action.
In particular, its $G$-invariant part $\A^*(f,G)^G$ is isomorphic as a $\ZZ/2\ZZ$-graded algebra 
to the Hochschild cohomology $\ccHH^*(\MF_G(f))$ of $\MF_G(f)$ equipped with the cup product
\cite[Theorem~3.1 and Theorem~3.4]{S20}. 
We shall recall the product structure of $\A^*(f,G)$ \cite[Section~3]{S20}. 

Define the $n$-th $\ZZ$-graded {\em Clifford algebra} $\Cl_n$ as the quotient algebra of 
\[
\CC\langle\theta_1,\ldots,\theta_n,{\partial_{\theta_1}},\ldots,{\partial_{\theta_n}}\rangle
\]
modulo the ideal generated by 
\[
\theta_i\theta_j=-\theta_j\theta_i,\quad{\partial_{\theta_i}}{\partial_{\theta_j}}=-{\partial_{\theta_j}}{\partial_{\theta_i}},\quad{\partial_{\theta_i}}\theta_j=-\theta_j{\partial_{\theta_i}}+\delta_{ij},
\]
where $\theta_i$ is of degree $-1$ and ${\partial_{\theta_i}}$ is of degree $1$.
For $I\subseteq\{1,\ldots,n\}$ write
\begin{equation}
\partial_{\theta_{I}}:=\prod_{i\in I}\partial_{\theta_i}, 
\quad {\theta_{I}}:=\prod_{i\in I}{\theta_i},
\end{equation}
where in both cases  the multipliers are taken in increasing order of the indices. 
The subspaces $\CC[{\theta}]=\CC[\theta_1,\ldots,\theta_n]$ and $\CC[\partial_{\theta}]=\CC[\partial_{\theta_1},\ldots,\partial_{\theta_n}]$ of $\Cl_n$ have the left $\ZZ$-graded $\Cl_n$-module structures via the isomorphisms  
\[
\CC[{\theta}]\cong \Cl_n/\Cl_n\langle{\partial_{\theta_1}},\ldots,{\partial_{\theta_n}}\rangle, \quad \CC[\partial_{\theta}]\cong \Cl_n/\Cl_n\langle{{\theta_1}},\ldots,{{\theta_n}}\rangle. 
\]  

Write $\CC[x_1,\dots,x_n,y_1,\ldots, y_n]$ as $\CC[\bx,\by]$ and 
$\CC[x_1,\dots,x_n,y_1,\ldots, y_n, z_1,\dots, z_n]$ as $\CC[\bx,\by,\bz]$. 
For each $1 \le i \le n$, there is a map
\begin{eqnarray}\label{deltai}
\nabla^{\bx\to (\bx,\by)}_i: \CC[\bx]\to \CC[\bx,\by],\qquad \nabla_i(p):=\frac{l_i(p)-l_{i+1}(p)}{x_i-y_i}.
\end{eqnarray}
where $l_i(p):=p(y_1,\ldots,y_{i-1},x_i,\ldots,x_n)$, $l_1(p) = p(\bx)$ and $l_{n+1}(p) = p(\by)$ \cite[Section~3.1.1]{S20}. 
They are called the \textit{difference derivatives}, whose key property is the following: 
\begin{equation}\label{diffder}
\sum_{i=1}^n(x_i-y_i)\nabla_i(p)=p(\bx)-p(\by).
\end{equation}

The difference derivatives can be applied consecutively. In particular, we shall use 
$\nabla_i^{\by\to(\by,\bz)}\nabla_j^{\bx\to(\bx,\by)}(p)$, which is an element of $\CC[\bx,\by,\bz]$.  
For an $\CC$-algebra homomorphism $\psi: \CC[\bx] \to \CC[\bx]$, write $\nabla_i^{\bx\to(\bx,\psi(\bx))}(p) := \left.\nabla_i^{\bx\to(\bx,\by)}(p) \right|_{\by = \psi(\bx)} \in \CC[\bx]$.

Now we are ready to describe the product structure of $\A^*(f,G)$. For each pair $(g,h)$ of elements in $G$, define the class $\sigma_{g,h}\in \Jac(f^{gh})$ as follows.
\begin{itemize}
\item
If $d_{g,h}:=\frac{1}{2}(d_g +d_h - d_{gh})$ is not a non-negative integer, set $\sigma_{g,h} = 0$.
\item
If $d_{g,h}$ is a non-negative integer, define $\sigma_{g,h}$ to be the class of the coefficient of $\partial_{\theta_{I_{gh}^c}}$ in the expression 
\begin{equation}\label{eq: sigma g,h}
\Small{
\frac1{d_{g,h}!}\,\LL\left(\left(\left\lfloor\rmH_f(\bx,g(\bx),\bx)\right\rfloor_{gh}+\lfloor\rmH_{{f,g}}(\bx)\rfloor_{gh}\otimes1+1\otimes \lfloor \rmH_{{f,h}}(g(\bx))\rfloor_{gh}\right)^{d_{g,h}}\otimes \partial_{\theta_{I_g^c}}\otimes\partial_{\theta_{h}^c}\right)
}
\end{equation} 
where 
\begin{itemize}
\item[(1)]
$\rmH_f(\bx,g(\bx),\bx)$ is the element of $\CC[\bx]\otimes\CC[\theta]^{\otimes2}$ defined as the restriction to the set 
$\{\by=g(\bx),\, \bz=\bx\}$ of the following element of  $\CC[\bx,\by,\bz]\otimes\CC[\theta]^{\otimes 2}$ 
\begin{equation}\label{del-2}
\rmH_{f}(\bx,\by,\bz):=\sum_{1\leq j\leq i\leq n} \nabla^{\by\to(\by,\bz)}_j\nabla^{\bx\to(\bx,\by)}_i(f)\,\theta_i\otimes \theta_j;
\end{equation}
\item[(2)]
$\rmH_{{f,g}}(\bx)$ is the element of $\CC[\bx]\otimes \CC[\theta]$ is given by 
\begin{eqnarray}
\rmH_{{f,g}}(\bx):=\sum_{{i,j\in I_{g}^c,\,\,  j<i}}\frac{1}{1-g_j}\nabla^{\bx\to(\bx,\bx^g)}_j\nabla^{\bx\to(\bx,g(\bx))}_i(f)\,\theta_j\,\theta_i,
\end{eqnarray}
where $\bx^g$ is defined as $(\bx^g)_i=x_i$ if $i\in I_g$ and $(\bx^g)_i=0$ if $i\in I_g^c$;
\item[(3)]
$\left\lfloor\rm- \right\rfloor_{gh}:\CC[\bx]\otimes V\longrightarrow \Jac(f^{gh})\otimes V$ for 
$V=\CC[\bx]\otimes \CC[\theta]^{\otimes 2}$ or $V=\CC[\bx]\otimes \CC[\theta]$
is a $\CC$-linear map defined as the extension of the quotient map $\CC[\bx]\longrightarrow \Jac(f^{gh})$;
\item[(4)]
the $d_{g,h}$-th power  in Equation~\eqref{eq: sigma g,h} is computed with respect to the natural product on  $\CC[\bx]\otimes \CC[\theta]\otimes \CC[\theta]$;
\item[(5)]
$\LL$ is the $\CC[\bx]$-linear extension of the degree zero map $\CC[{\theta}]^{\otimes2}\otimes \CC[\partial_{\theta}]^{\otimes2}\to\CC[\partial_{\theta}]$ defined by
\begin{equation}\label{mu}
 p_1(\theta)\otimes p_2(\theta)\otimes q_1(\partial_\theta)\otimes q_2(\partial_\theta)\mapsto(-1)^{|q_1||p_2|}p_1(q_1)\cdot p_2(q_2)
\end{equation}
where $p_i(q_i)$ denotes the action of $p_i(\theta)$ on $q_i(\partial_\theta)$ via the $\Cl_n$-module structure on $\CC[\partial_\theta]$ defined above and $\cdot$ is the natural product in $\CC[\partial_{\theta}]$. 
\end{itemize}
\end{itemize}
Then the product of $\A^*(f,G)$ is given by
\begin{equation}\label{eq: HH cup}
[\phi(\bx)]\xi_g\cup [\psi(\bx)]\xi_{h} 
=[\phi(\bx)\psi(\bx) \sigma_{g,h}]\xi_{gh},\quad \phi(\bx), \psi(\bx)\in \CC[\bx].
\end{equation}

\subsection{Explicit description of $\ccHH^*(\MF_G(\fAp))$}
Assume the notation of Section~\ref{sec:symmetry group of affine cusp}.
For $i=1,2,3$ denote by $g_i$ a generator of $K_i \subseteq G$. We have $g_i^{n_i} = \id \in G$. Note that the elements $[\phi(\bx)] \xi_{g_i^a}$ have even parity in $\A^*(\fAp,G)$ for all $i,a$. By Proposition~\ref{prop:G} we have also $2j_G$ group elements left, giving odd parity sectors. These group elements have empty fixed locus.  

We have to compute the following products
\[
  \xi_{g_k^a} \cup \xi_{g_k^b}, \quad \xi_{g_k^a} \cup \xi_{g_l^b}
\]
in even sector and also
\[
  \xi_{g_k^a} \cup \xi_{h}, \quad \xi_h \cup \xi_{h'} \quad \text{ for } h,h' \in G, \ \Fix(h) = \Fix(h') = 0
\]
in the odd sector.

Consider one more piece of notation. For any $k \in 1,2,3$ and $g_k^a \neq \id$ there exists $\lambda \in \CC^* \backslash \{1\}$ such that 
\[
  g_k^a =
  \begin{cases}
  (1,\lambda,\lambda^{-1}), \quad & k=1,
  \\
  (\lambda,1,\lambda^{-1}), \quad & k=2,
  \\
  (\lambda,\lambda^{-1},1), \quad & k=3.
  \end{cases}
\]
Define the elements:
\[
  \tilde \xi_{g_k^a} := (1 - \lambda^{-1}) \xi_{g_k^a}.
\]
Recall that we consider $\A^*(\fAp,G)$ as a $\CC(c)$--module. Its structure is described by the following proposition.

\begin{proposition}\label{prop: HH of fAp}
For any $k,l,m$ such that $\{k,l,m\} = \{1,2,3\}$ and $h,h' \in G$, such that $\Fix(h) = \Fix(h') = 0$ we have in $\A^*(\fAp,G)$
\begin{align}
  & \tilde \xi_{g_k^a} \cup \tilde \xi_{g_k^b} 
  = \begin{cases}
    - \lfloor c^{-1} x_k \rfloor \cdot \tilde \xi_{g_k^{a+b}} \quad & \text{ if } g_k^{a+b} \neq \id,
    \\
    \lfloor - a_l' a_m' x_l^{a_l'-2}x_m^{a_m'-2}  + c^{-2}x_k^2 \rfloor \cdot \xi_\id \quad & \text{ if } g_k^{a+b} = \id.
    \end{cases}
  \\
  & \tilde \xi_{g_k^a} \cup \tilde \xi_{g_l^b} 
    = \begin{cases} 
	0  &\text{ if } \quad \Fix(g_k^ag_l^b) = 0,
	\\
	\lfloor a_m'(a_m'-1) x_m^{a_m'-2} \rfloor \cdot \tilde \xi_{g_k^ag_l^b} \quad &\text{ otherwise}.
      \end{cases}
\end{align}
\end{proposition}

To prove the proposition we should prepare some computations.
We have
\begin{align}
\rmH_\fAp(\bx,\by,\bz) &= \sum_{k=1}^3 \frac{1}{y_k-z_k} \left(\frac{x_k^{a_k'}-y_k^{a_k'}}{x_k-y_k} - \frac{x_k^{a_k'}-z_k^{a_k'}}{x_k-z_k}\right) \theta_k \otimes\theta_k
\\
&- c^{-1} x_3 \cdot \theta_2 \otimes\theta_1 - c^{-1} y_2 \cdot \theta_3 \otimes\theta_1 - c^{-1} z_1 \cdot \theta_3 \otimes\theta_2.
\end{align}

To simplify the formulae in what follows we use the notation. For a given group element $g$ let the polynomials $A_{kl} \in \CC[\bx]$ be fixed by the equality
\[
  \rmH_\fAp(\bx,g(\bx),\bx) = \sum_{i \le j} A_{ji} \ \theta_j \otimes \theta_i.
\]
Obviously, these polynomials depend on $g$, however this dependence will be clear in the context and therefore dropped in the notation.

We should also compute $\rmH_{\fAp,g}(\bx)$ for different group elements $g$. For $g = (g_1,g_1^{-1},1)$ we have
\begin{align}
  &\rmH_{\fAp,g}(\bx) = \frac{1}{1-g_1} \nabla^{\bx \to (\bx,\bx^g)}_1 \nabla^{\bx \to (\bx,g(\bx))}_2(\fAp) \theta_1\theta_2
  \\
  &\quad = \frac{1}{1-g_1} \nabla^{\bx \to (\bx,\bx^g)}_1 \frac{1}{x_2(1-g_1^{-1})} \big( (g_1x_1)^{a_1'} + x_2^{a_2'} + x_3^{a_3'} - c^{-1} g_1 x_1x_2x_3
  \\
  &\quad - ((g_1x_1)^{a_1'} + (g_1^{-1}x_2)^{a_2'} + x_3^{a_3'} - c^{-1} x_1x_2x_3) \big) \theta_1\theta_2
  \\
  &\quad = \frac{1}{1-g_1} \nabla^{\bx \to (\bx,\bx^g)}_1 c^{-1}x_1x_3 \frac{1-g_1}{1-g_1^{-1}} \theta_1\theta_2 
  \\
  &\quad = \frac{c^{-1}}{1-g_1^{-1}} x_3 \theta_1\theta_2.
\end{align}
Similarly for $g=(g_1,1,g_1^{-1})$
\begin{align}
  &\rmH_{\fAp,g}(\bx) = c^{-1} \frac{1}{1 - g_1^{-1}} x_2 \theta_1 \theta_3,
\end{align}
and $g=(1,g_1,g_1^{-1})$
\begin{align}
  &\rmH_{\fAp,g}(\bx) = c^{-1} \frac{1}{1 - g_1^{-1}} x_1 \theta_2 \theta_3.
\end{align}

\begin{proof}[Proof of Proposition~\ref{prop: HH of fAp}]
~
\\
  {\bf Case 1: $\xi_{g_k^a} \cup \xi_{g_k^b}$, $g_k^{a+b} \neq \id$.} Assume $g_k = g_3$, the proof will be the same for $g_2$ and $g_1$. Then for some $\lambda_1,\lambda_2 \in \CC^*$ we have $g_k^a = (\lambda_1,\lambda_1^{-1},1)$ and $g_k^b = (\lambda_2,\lambda_2^{-1},1)$. We have $d_{g_k^a,g_k^b} = 1$ and following the technique of Shklyarov we need to find the coefficient of $\p_{\theta_1}\p_{\theta_2}$ in the expression
  \begin{align}
    \Big( \lfloor \rmH_\fAp(\bx,g_k^a(\bx),\bx) \rfloor_{g_k^{a+b}} &+ \lfloor \rmH_{\fAp,g_k^a}(\bx) \rfloor_{g_k^{a+b}} \otimes 1 
    \\
    &+ 1 \otimes \lfloor \rmH_{\fAp,g_k^b}(g_k^a(\bx)) \rfloor_{g_k^{a+b}}\Big) \cdot \p_{\theta_1}\p_{\theta_2} \otimes \p_{\theta_1}\p_{\theta_2}.
  \end{align}
  We have 
  \[
    \lfloor \rmH_{\fAp,g_k^a}(\bx) \rfloor_{g_k^{a+b}} = \frac{c^{-1}}{1-\lambda_1^{-1}} x_3 \theta_1\theta_2,  
    \quad
    \lfloor \rmH_{\fAp,g_k^b}(g_k^b(\bx)) \rfloor_{g_k^{a+b}} = \frac{c^{-1}}{1-\lambda_2^{-1}} x_3 \theta_1\theta_2,
  \]
  and by the simple combinatorics we have to compute the coefficient of $\p_{\theta_1}\p_{\theta_2}$ in
  \begin{align}
    \Big( A_{21} \cdot \theta_2 \otimes \theta_1 + \frac{c^{-1}}{1-\lambda_1^{-1}} x_3  \cdot \theta_1\theta_2\otimes 1 
    + \frac{c^{-1}}{1-\lambda_2^{-1}} x_3 \cdot 1 \otimes \theta_1\theta_2 \Big) \cdot \p_{\theta_1}\p_{\theta_2} \otimes \p_{\theta_1}\p_{\theta_2}.
  \end{align}
  Therefore we have
  \[
    \sigma(g_k^a,g_k^b) = - c^{-1} x_3 \left( -1 + \frac{1}{1-\lambda_1^{-1}} + \frac{1}{1-\lambda_2^{-1}}\right) = - c^{-1} x_3 \frac{1 - (\lambda_1\lambda_2)^{-1}}{(1-\lambda_1^{-1})(1-\lambda_2^{-1})}.
  \]
  \\{\bf Case 2: $\xi_{g_k^a} \cup \xi_{g_k^{-a}}$.} Assume $g_k = g_3$, the proof will be the same for $g_2$ and $g_1$. Then for some $\lambda_1 \in \CC^*$ we have $g_k^a = (\lambda_1,\lambda_1^{-1},1)$ and $g_k^{-a} = (\lambda_1^{-1},\lambda_1,1)$. We have $d_{g_k^a,g_k^{-a}} = 2$ and we need to find the coefficient of $1$ in the expression
  \begin{align}
    \frac{1}{2}\Big( \lfloor \rmH_\fAp(\bx,g_k^a(\bx),\bx) \rfloor_{\id} &+ \lfloor \rmH_{\fAp,g_k^a}(\bx) \rfloor_{\id} \otimes 1 
    \\
    &+ 1 \otimes \lfloor \rmH_{\fAp,g_k^b}(g_k^a(\bx)) \rfloor_{\id}\Big)^2 \cdot \p_{\theta_1}\p_{\theta_2} \otimes \p_{\theta_1}\p_{\theta_2}.
  \end{align}
  Using the expressions of $\rmH_{\fAp,g_k^a}$ computed above $\sigma(g_k^a,g_k^{-a})$ is the coefficient of $1$ in
  \begin{align}
    & \Big( A_{11}\theta_1\otimes\theta_1 \cdot A_{22} \theta_2\otimes\theta_2 + \frac{c^{-1}}{1 - \lambda_1^{-1}}x_3 \theta_1\theta_2 \otimes  \frac{c^{-1}}{1 - \lambda_1}x_3 \theta_1\theta_2 \Big) \cdot \p_{\theta_1}\p_{\theta_2} \otimes \p_{\theta_1}\p_{\theta_2}
    \\
    & = \left( - \frac{a_1' x_1^{a_1'-2}}{\lambda_1-1}\frac{a_2' x_2^{a_2'-2}}{\lambda_1^{-1}-1}  + \frac{c^{-2}}{1 - \lambda_1^{-1}} \frac{x_3^2}{1 - \lambda_1}   \right) \theta_1\theta_2\otimes\theta_1\theta_2\cdot \p_{\theta_1}\p_{\theta_2} \otimes \p_{\theta_1}\p_{\theta_2}
  \end{align}
  Therefore we have
  \[
    \sigma(g_k^a,g_k^{-a}) = \frac{1}{1 - \lambda_1^{-1}} \frac{1}{1 - \lambda_1} \left( - a_1' a_2' x_1^{a_1'-2}x_2^{a_2'-2}  + c^{-2}x_3^2\right).
  \]
  \\
  {\bf Case 3: $\xi_{g_k^a} \cup \xi_{g_l^b}$.} If $\Fix(g_k^ag_l^b) = 0$ we have $\xi_{g_k^a} \cup \xi_{g_l^b} = 0$ by the parity counting. 
  Otherwise assume $k=3$ and $l=2$, the proof for all other cases will be similar. For some $\lambda_1 \in \CC^*$ we have $g_k^a = (\lambda_1,\lambda_1^{-1},1)$ and $g_l^b = (\lambda_1^{-1},1,\lambda_1)$. $d_{g_k^a,g_l^b} = 1$ and we need to find the coefficient of $\p_{\theta_2}\p_{\theta_3}$ in the expression
  \begin{align}
    \Big( \lfloor \rmH_\fAp(\bx,g_k^a(\bx),\bx) \rfloor_{g_k^{a+b}} &+ \lfloor \rmH_{\fAp,g_k^a}(\bx) \rfloor_{g_k^{a+b}} \otimes 1
    \\
    &+ 1 \otimes \lfloor \rmH_{\fAp,g_k^b}(g_k^a(\bx)) \rfloor_{g_k^{a+b}}\Big) \cdot \p_{\theta_1}\p_{\theta_2} \otimes \p_{\theta_1}\p_{\theta_3}.
  \end{align}
  The fixed locus of $g_k^ag_l^b$ is $\CC_{x_1}$ and we have
  \[
    \lfloor \rmH_{\fAp,g_k^a}(\bx) \rfloor_{g_k^{a}g_l^b} = \lfloor \rmH_{\fAp,g_k^b}(g_k^b(\bx)) \rfloor_{g_k^{a}g_l^b} = 0,
  \]
  from where $\sigma(\xi_{g_k^a},\xi_{g_k^b})$ is the coefficient of $\p_{\theta_2}\p_{\theta_3}$ in the expression
  \begin{align}
    \Big( A_{11} \cdot \theta_1 \otimes \theta_1 \Big) \cdot \p_{\theta_1}\p_{\theta_2} \otimes \p_{\theta_1}\p_{\theta_3},
  \end{align}
  We get
  \[
    \sigma(g_k^a,g_l^b) = \frac{\lambda_1}{\lambda_1-1} a_1'(a_1'-1) x_1^{a_1'-2}.
  \]  
\end{proof}

\begin{remark}
The general technique of Shklyarov is applied for the functions defined globally as we need to assume only the isolated singularity at the origin (recall Remark~\ref{remark: affine cusp singularities}). Due to this fact the technique of Shklyarov can be applied directly in the cases when $\chi_{A'} \ge 0$, but need to be refined in the other cases.
Exactly this problem with $A' = (2g+1,2g+1,2g+1)$ was addressed in Appendix~A of \cite{S20}. 
It was shown there that for cusp polynomials with $\chi_{A'} \le 0$, namely the pair $(\CC[[\bx]], \fAp)$ the technique of Shklyarov can still be applied. The same reasoning works for any set $A'$ we use.
\end{remark}

\subsection{Extension of $\ccHH^*(\MF_G(\fAp))$}\label{section: ext of HH}

Let $G$ be such that $j_G = 0$. For all $ g \in G \backslash \{\id\}$ consider the vectors $e_g := - c \cdot \tilde \xi_{g}$. Define also $e_{g_k^0} := \frac{1}{n_i} \lfloor x_k \rfloor \cdot \xi_{\id}$. Let $\overline{\A^*}(\fAp,G)$ be the free $\CC[c] \otimes_\CC \Jac(\fAp)$--module of rank $\mu_{A'}$ generated by $e_g$ for all $g \in G$ and $\lfloor c^{-1}x_1x_2x_3 \rfloor \cdot \xi_\id$. Denote by $\circ$ its product obtained by restriction of $\cup$--product. We need to show that $\circ$ is a $\CC[c]$--product.

For all $i = 1,2,3$ and $k = 1,\dots,n_i$ define:
\[
[x_{i,k}] := \sum_{l=0}^{n_i-1} \omega_i^{(k-1)l} \lfloor x_i^{[G/K_i]-1} \rfloor e_{g_i^l} \in \overline{\A^*}(f_{A'},G),
\]
for $\omega_i := \epi\left[1/ n_i\right]$.

\begin{proposition}\label{prop: structure constants of xijk}
The following equalities hold in $\overline{\A^*}(f_{A'},G)$.
\begin{align}
    & [x_{i,k}]^{\circ 2} = \lfloor x_i^{[G/K_i]} \rfloor x_{i,k}, \quad \forall i,k, \label{eq: xik xik}
    \\
    & [x_{i,k_1}] \circ [x_{i,k_2}] = 0 \quad \forall k_1 \neq k_2, \label{eq: xik1 xik2}
    \\
    & [x_{i,k_i}] \circ [x_{j,k_j}] = \frac{1}{n_in_j} \lfloor x_ix_j \rfloor \xi_\id, \quad \forall i \neq j. \label{eq: xijki xjkj}
\end{align}

\end{proposition}
\begin{proof}
 We have
   \begin{align}
    [x_{i,k}]^{\circ 2} & = \sum_{l'=0}^{n_i-1} \omega_i^{(k-1)l'}\sum_{l''=0}^{n_i-1} \omega_i^{(k-1)l''} \lfloor x_i^{[G/K_i]-1} \rfloor \lfloor x_i^{[G/K_i]-1} \rfloor  \lfloor x_i \rfloor e_{g_i^{l'+l''}} 
    \\
    &= \lfloor x_i^{[G/K_i]} \rfloor \sum_{l'=0}^{n_i-1} \omega_i^{(k-1)l'}\sum_{l''=0}^{n_i-1} \omega_i^{(k-1)l''} \lfloor x_i^{[G/K_i]-1} \rfloor  e_{g_i^{l'+l''}}
    \\
    &= \lfloor x_i^{[G/K_i]} \rfloor x_{i,k}.
  \end{align}
  Similarly
\begin{align}
    [x_{i,k_1}] \circ [x_{i,k_2}] & = \sum_{l'=0}^{n_i-1} \omega_i^{(k_1-1)l'}\sum_{l''=0}^{n_i-1} \omega_i^{(k_2-1)l''} \lfloor x_i^{[G/K_i]-1} \rfloor \lfloor x_i^{[G/K_i]} \rfloor e_{g_i^{l'+l''}} 
    \\
    & = \sum_{l',l''=0}^{n_i-1}\omega_i^{(k_1-1)l' + (k_2-1)l''} \lfloor x_i^{[G/K_i]-1} \rfloor \lfloor x_i^{[G/K_i]} \rfloor e_{g_i^{l'+l''}} =0. 
\end{align}
Now compute the mixed products $[x_{i,k_1}] \circ [x_{j,k_2}]$ with $i \neq j$
\begin{align}
[x_{1,k_1}] \circ [x_{2,k_2}] &= e_{g_1^0} \circ e_{g_2^0} = \frac{1}{n_1n_2} \lfloor x_1x_2 \rfloor \xi_\id.
\end{align}
Remaining products $[x_{1,k_1}] \circ [x_{3,k_2}]$ and $[x_{2,k_1}] \circ [x_{3,k_2}]$ are computed similarly.
    
\end{proof}

\begin{corollary}\label{proposition: oJac structure}
There exists a free $\CC[c]$--module $\overline{\A^*}(\fAp,G) \hookrightarrow \A^*(f_{A'},G)  $ of rank $\mu_{A'}$, such that
\begin{description}
  \item[(a)] $\overline{\A^*}(\fAp,G)$ is a Frobenius algebra,
  \item[(b)] the quotient $\overline{\A^*}(\fAp,G) \mid_{c=0} \ := \overline{\A^*}(\fAp,G) / (c)$ is a $\mu_{A'}$--dimensional $\CC$--algebra, containing $\overline{\Jac}(\fAp)$ as a subalgebra
  \item[(c)] the $G$--invariant submodule $\left( \overline{\A^*}(\fAp,G) \mid_{c=0} \right)^G$ is a $\mu_{A'}$--dimensional $\CC$--algebra uniquely determined by $(\fAp,G)$ up to isomorphism. It is isomorphic to
    \begin{equation}
    \CC[z_{1,1},\dots, z_{3,n_3}]\left/\left(z_{i, k}z_{j,l},a_i z_{i,k}^{a_i}-a_j z_{j,l}^{a_j};i,j=1,2,3, k=1,\dots, n_i, l=1,\dots, n_j\right)\right..
    \end{equation}
\end{description}

\end{corollary}
\begin{proof}
  By the definition we have an embedding
  \[
   \overline{\A^*}(f_{A'},G) \hookrightarrow \A^*(f_{A'},G).
  \]
  The products $e_g \circ e_h$ can be computed via Proposition~\ref{prop: HH of fAp}. It follows immediately from it that the structure constants of $\overline{\A^*}(f_{A'},G)$ are polynomial in $c$. We get a Frobenius algebra structure because $\A^*(f_{A'},G)$ is Frobenius algebra. Denote the product of it by $\circ$. This gives (a).
  
  Part (b) is straightforward. 
  By the definition, the action of $h = (h_1,h_2,h_3) \in G$ on $\tilde \xi_{g_1}$ reads $h(\tilde \xi_{g_1}) = (h_2h_3)^{-1}\tilde \xi_{g_1} = h_1 \tilde \xi_{g_1}$ and similarly for $e_{g_2},e_{g_3}$. Therefore $[x_{i,k}]$ above is $G$--invariant. 
  It is clear that all such elements generate $\left( \overline{\A^*}(\fAp,G) \mid_{c=0} \right)^G$ as a $\CC$--algebra.

  The mixed products $\lfloor x_i x_j \rfloor$ vanish in $\overline{\Jac}(\fAp) \mid_{c=0}$. From Eq.\eqref{eq: xijki xjkj} above we have $[x_{i,k_i}] \circ [x_{j,k_j}] = 0$ in $\left( \overline{\A^*}(\fAp,G) \mid_{c=0} \right)^G$. From Eq.\eqref{eq: xik xik} we get
    \begin{equation}
    [x_{i,k}]^{\circ a_i} = n_i [x_i^{a_i'}] \cdot \xi_\id = \frac{a_i'}{a_i |G|} [x_i^{a_i'}] \cdot \xi_\id.
    \end{equation}

  Because $a_i' \lfloor x_i^{a_i'} \rfloor = a_j' \lfloor x_j^{a_j'} \rfloor$ in $\overline{\Jac}(\fAp)$ we have $a_i[x_{i,k}]^{\circ a_i} = a_j[x_{j,k}]^{\circ a_j}$ in $\left( \overline{\A^*}(\fAp,G) \mid_{c=0} \right)^G$.
  This concludes the proof.
  
\end{proof}

\section{Frobenius manifold of a Landau--Ginzburg orbifold}\label{sec:symmetry group of affine cusp}

Fix a pair $(\fAp,G)$. Let $M_\fAp^\infty$ be the Frobenius manifold of $\fAp$ with a primitive form at infinity as in Section~\ref{section: frobenius manifolds}. It gives the Frobenius structure on $\mathcal{S} = \CC^{\mu_{A'}} \times M^s$ for $M^s \cong \{z \in \CC\backslash\{0\}~|~|\! |z|\! |<\epsilon\}$.
Denote $\mathcal{S}_{f_{A'},G} := \CC^{\mu_{A}-1} \times M^s$. 

\begin{definition}
 We call $M_{f_{A'},G}^{\infty}$ {\em the Frobenius manifold of the pair $(\fAp,G)$ with the primitive form at infinity} if it gives a Frobenius structure on $\mathcal{S}_{(\fAp,G)}$, satisfying the following axioms below.
\end{definition}

\subsection{Axioms}\label{section: axioms}
\begin{description}
    \item[(coordinates axiom)] $M_{\fAp,G}^\infty$ has the flat coordinates 
    \begin{align}
        v_{\id,k}, \ 1\le k \le \mu_A, \quad \text{and} \quad  v_{g_i^l,k}, \ 1 \le l \le n_i-1, \ 1 \le k \le a_i-1,
    \end{align}
    for $g_i$ being a generator of $K_i \subseteq G$. 
    
    The vector $\p/\p v_{\id,1}$ is the unit of the Frobenius manifold,
    \item[(quasihomogeneity axiom)] $M_{\fAp,G}^\infty$ is quasihomogeneous of conformal dimension $1$ with respect to the Euler field 
    \[
        E = v_{\id,1} \frac{\p}{\p v_{\id,1}} + \sum_{i=1}^3\sum_{l=1}^{n_i-1}\sum_{k=1}^{a_i-1} \frac{a_i-k}{a_i} v_{g_i^l,k} \frac{\p}{\p v_{g_i^l,k}} + \chi_A \frac{\p}{\p v_{\id,\mu_A}}
    \] 
    \item[(expansion axiom)] potential of $M_{\fAp,G}^\infty$ has a series expansion in $v_{\id,1},\dots,v_{\id,\mu_A-1}$, $v_{g_i^l,k}$ and $\exp(|G|v_{\mu_A})$,
    \item[(extended $\Jac$ axiom)] the tangent space $T_0 M_{\fAp,G}^\infty$, restricted at $v_{\id,1}=\dots=v_{\id,\mu_A-1}=0$ and $v_{g_i^l,k} = 0$, is isomorphic to $\overline{\A^*}(\fAp,G)$ as Frobenius $\CC[c]$--algebra with $c = \exp(|G|v_{id,\mu_A})$
    \item[($G$-grading axiom)] for every fixed $1 \le i \le 3$ the potential of $M_{\fAp,G}^\infty$ is invariant under the change of the variables 
    $v_{g_i^l,k} \to g_i^l v_{g_i^l,k}$, $1 \le l \le n_i-1$.
    
    It follows from this axiom together with extended $\Jac$ axiom that the product of $M^\infty_{\fAp,G}$ is $G$--graded. In particular,
    \[
        M^\infty_{\fAp,\id} := M^\infty_{\fAp,G} \mid _{v_{g_i^l,k} = 0}.
    \]
    is a Frobenius submanifold.
    
    \item[(invariant sector axiom)] 
    there is a Frobenius manifold isomorphism ${ (M_{\fAp}^\infty)^G \cong M_{\fAp,\id}^\infty }$.
    \item[($\Aut$--invariance axiom)] for every group automorphism $\phi: G \to G$, the potential of $M_{\fAp,G}^\infty$ is invariant with respect to the change of the variables $v_{g_i^l,k} \to v_{\phi(g_i^l),k}$.
\end{description}
\begin{remark}
All the axioms above are essential from the point of view of Landau-Ginzburg orbifolds. 
Coordinate axiom, invariant sector axiom and expansion axiom are expected by all Frobenius manifolds of Landau-Ginzburg orbifolds. 
Extended $\Jac$ axiom is specific for the cusp polynomials, having one--dimensional space of marginal deformations. 
Quasihomogeneity axiom could be reformulated by saying that the quasihomogeneity of the Frobenius manifold should agree with the quasihomogeneity of $\overline{A^*}(\fAp,G)$. 
$G$--grading axiom just requires the product of $M_{\fAp,G}^\infty$ to respect the $G$--grading.
$\Aut$--invariance axiom is specific for $(\fAp,G)$. It reflects the fact that the subgroups $K_i \subseteq G$ don't have any canonical generator whereas some generator is fixed in our notation.
\end{remark}

Main theorem of this paper is the following.
\begin{theorem}\label{theorem: main}
  Let $M_{(f_{A'},G)}^{\infty}$ be the Frobenius manifold of the pair $(f_{A'},G)$ as above.  
  Then there is a Frobenius manifolds isomorphism:
  \begin{equation}
    M_{\PP^1_{A,\Lambda}} \cong M_{(f_{A'},G)}^{\infty}.
  \end{equation}
  In particular the potential $\F_{(\fAp,G)}^\infty$ coincides with the potential of genus $0$ orbifold Gromov--Witten theory of $\PP^1_A$ after a suitable choice of coordinates.
\end{theorem}
\begin{corollary}
    The Frobenius manifold $M_{(f_{A'},G)}^{\infty}$ is uniquely determined by the axioms. 
\end{corollary}

It follows immediately from its definition that for $\chi_{A'} = 0$ (respectively $\chi_{A'} > 0$) we have $\chi_{A} = 0$ (respectively $\chi_{A} > 0$) for any symmetry group allowed. 
We list all possible pairs $A'$ and $G$ with $\chi_{A'} \ge 0$ in Table~\ref{tab:SES list} and Table~\ref{tab:cusp sing list} below.

\begin{remark}
If both $A'$ and $A$ are of length $3$, it turns out that we have a isomorphism $M_{(f_{A'},G)}^{\infty} \cong M_{f_{A}}^{\infty}$.
If this holds we write the type of the corresponding singularity in the last column. Note that one should not mix up this isomorphism with LG-LG mirror isomorphism because both Frobenius manifolds represent the B-models.
\end{remark}
{
\renewcommand\arraystretch{1.5}
\begin{table}[h]
\caption{\label{tab:SES list} Classification of $(\fAp,G)$ with $\chi_{A'} = 0$.}
\begin{tabular}{ c | c | c | c | c}
    $A^\prime$ & type of $\fAp$ & $G$ & $A$ & type of $f_A$
    \\
    \hline
    $(3,3,3)$ & $\widetilde E_6$ & $\langle (\epi[0],\epi[\frac{1}{3}],\epi[\frac{2}{3}]) \rangle$ & $(3,3,3)$ & $\widetilde E_6$
    \\
    $(4,4,2)$ & $\widetilde E_7$ & $\langle (\epi[0],\epi[\frac{1}{2}],\epi[\frac{1}{2}]) \rangle$ & $(4,4,2)$ & $\widetilde E_7$
    \\
    $(4,4,2)$ & $\widetilde E_7$ & $\langle (\epi[\frac{1}{4}],\epi[\frac{3}{4}],\epi[0]) \rangle$ & $(2,2,2,2)$ & -
    \\
    $(4,4,2)$ & $\widetilde E_7$ & $\langle (\epi[0],\epi[\frac{1}{2}],\epi[\frac{1}{2}]),(\epi[\frac{1}{2}],\epi[0],\epi[\frac{1}{2}]) \rangle$ & $(2,2,2,2)$ & -
    \\
    $(6,3,2)$ & $\widetilde E_8$ & $\langle (\epi[\frac{1}{2}],\epi[0],\epi[\frac{1}{2}]) \rangle$ & $(3,3,3)$ & $\widetilde E_6$
    \\
    $(6,3,2)$ & $\widetilde E_8$ & $\langle (\epi[\frac{1}{3}],\epi[\frac{2}{3}],\epi[0]) \rangle$ & $(2,2,2,2)$ & -
\end{tabular}
\end{table}
}

{
\renewcommand\arraystretch{1.5}
\begin{table}[h]
\caption{\label{tab:cusp sing list} Classification of $(\fAp,G)$ with $\chi_{A'} > 0$.}
\begin{tabular}{ c | c | c | c | c}
    $A^\prime$ & type of $\fAp$ & $G$ & $A$ & type of $f_A$
    \\
    \hline
    $(2,3,5)$ & $\hat E_8$ & none & - & -
    \\
    $(2,3,4)$ & $\hat E_7$ & $\langle (\epi[\frac{1}{2}],\epi[0],\epi[\frac{1}{2}]) \rangle$ & $(3,3, 2)$ & $\hat E_6$
    \\
    $(2,3,3)$ & $\hat E_6$ & $\langle (\epi[0], \epi[\frac{1}{3}],\epi[\frac{2}{3}]) \rangle$ & $(2,2,2)$ & $\hat D_4$
    \\
    $(2,2,2k)$ & $\hat D_{2k+2}$ & $\langle (\epi[0],\epi[\frac{1}{2}],\epi[\frac{1}{2}]),(\epi[\frac{1}{2}],\epi[0],\epi[\frac{1}{2}]) \rangle$ & $(1,k,k)$ & $\hat A_{2k-1}$ 
    \\
    $(2,2,2k)$ 
    &  $\hat D_{2k+2}$ & $\langle (\epi[0], \epi[\frac{1}{2}],\epi[\frac{1}{2}]) \rangle$ & $(2,2, k)$ & $\hat D_{k+2}$
    \\
    $(2,2,2k)$ 
    &  $\hat D_{2k+2}$ & $\langle (\epi[\frac{1}{2}],\epi[\frac{1}{2}],\epi[0]) \rangle$ & $(1,2k,2k)$ & $\hat A_{4k-1}$ 
    \\
    $(2,2,2k+1)$ & $\hat D_{2k+3}$ & $\langle (\epi[\frac{1}{2}],\epi[\frac{1}{2}],\epi[0]) \rangle$ & $(1, 2k+1,2k+1)$ & $\hat A_{4k+1}$ 
    \\
    $(1,k\cdot m,l\cdot m)$ & $\hat A_{m(k+l)-1}$ & $\langle (\epi[0],\epi[\frac{1}{m}],\epi[\frac{m-1}{m}]) \rangle$ & $(1, k, l)$ & $\hat A_{k+l-1}$ 
\end{tabular}
\end{table}
}
\section{Proof of the main Theorem}\label{section: proof}

In what follows let $\F_G = \F_G(\bv)$ be the potential of $M_{(\fAp,G)}^\infty$. In this section we show that $\F_G$ satisfies all conditions of Theorem~\ref{theorem: yuuki--GW}.

Due to extended $\Jac$ axiom we can choose the coordinates on $M_{\fAp,G}^\infty$ such that
in $\overline{\A^*}(f_{A'},G)$ the coodinate $v_{\id,1}$ is dual to $\lfloor c^{-1}x_1x_2x_3\rfloor \cdot \xi_\id$, $v_{\id,\mu_A}$ is dual to $e_{\mu_{A'}}$ and $v_{g_i^j,l}$ is dual to $\left(e_{g_i^j} \right)^{\circ l}$.

\subsection{Conditions (i), (ii), (iii) and (v)}\label{section: MS proof part 1}
Consider the following change of variables $\bt = \bt(\bv)$:
\begin{equation}\label{eq: change of variables t=t(v)}
\begin{aligned}
  & t_{1} = v_{\id,1}, \quad t_{\mu_A} = v_{\id,\mu_A}/|G|,
  \\
  & t_{(i,k),j} = \sum_{l = 0}^{n_i-1} \omega_i^{(k-1)l} v_{g_i^l,j}, \quad \quad 1 \le i \le 3, \ 1 \le k \le n_i, \ 1 \le j \le a_i-1,
\end{aligned}
\end{equation}
where $\omega_i := \epi[1/n_i]$. Denote $\F(\bt) := |G| \F_{G}(\bv(\bt))$.

The coordinates $\bt$ introduced above satisfy the following properties. 
They are dual to the following generators of the algebra $\overline{\A^*}(\fAp,G)$ (recall Section~\ref{section: ext of HH}):
\begin{align*}
& t_{1} \text{: the flat coordinate dual to } [1] \text{ at the limit } \bs=s_{\mu_{A'}}=0. 
\\
& t_{(i,k),j} \text{: the flat coordinate dual to } [x_{i,k}]^j \text{ at the limit } \bs=s_{\mu_{A'}}=0. 
\\
& t_{\mu_A}  \text{: the flat coordinate dual to } \frac{(s_{\mu_{A'}})^{-1}[x_1x_2x_3]}{|G|}=\frac{a'_i[x_i]^{a'_i}}{|G|} \text{ at the limit } \bs=s_{\mu_{A'}}=0. 
\end{align*}

In the coordinates introduced condition (i) follows from coordinate and quasihomogeneity axioms of $M_{\fAp,G}^\infty$. 
Also condition (v) follows by Propositions~\ref{proposition: oJac structure}.

To show condition (ii) we compute explicitly the pairing using the Frobenius algebra property.

\begin{lemma}\label{lemma: pairingMain}
We have 
\begin{equation}
\eta\left(\frac{\p}{\p t_{1}}, \frac{\p}{\p t_{\mu_A}}\right)=1.
\end{equation}
\end{lemma}
\begin{proof}
It follows from condition (a) of $M_{\fAp,G}^\infty$, Lemma~4.4 in \cite{ist:2} and Lemma~4.6 in \cite{st:1}.
\end{proof}
\begin{lemma}\label{lemma: pairingTwist}
For all $i=1,2,3$ and $k=1,\dots, n_i$, we have 
\begin{equation}
\eta\left(\frac{\p}{\p t_{(i,k),j}}, \frac{\p}{\p t_{(i,k),a_i-j}}\right)=\frac{1}{a_i}.
\end{equation}
\end{lemma}
\begin{proof}
It is enough to calculate the pairing at the limit ${\bf t}=e^{t_{\mu_A}}=0$.
There we have 
\begin{eqnarray*}
& &\left.\eta\left(\frac{\p}{\p t_{(i,k),j}}, \frac{\p}{\p t_{(i,k),a_i-j}}\right)\right\vert_{{\bf t}=e^{t_{\mu_A}}=0}\\
&=&\left.\eta\left(\frac{\p}{\p t_1},\frac{\p}{\p t_{(i,k),j}}\circ \frac{\p}{\p t_{(i,k),a_i-j}}\right)\right\vert_{{\bf t}=e^{t_{\mu_A}}=0}\\
&=&\frac{1}{a_i}\cdot\left.\eta\left(\frac{\p}{\p t_1},\frac{\p}{\p t_{\mu_A}}\right)\right\vert_{{\bf t}=e^{t_{\mu_A}}=0}\\
&=& \frac{1}{a_i},
\end{eqnarray*}
since $x_{i,k}^{j}\cdot x_{i,k}^{a_i-j}=\frac{1}{a_i}\cdot \frac{a'_ix_i^{a'_i}}{|G|}$ at the limit ${\bf t}=e^{t_{\mu_A}}=0$.
\end{proof}



\subsection{Conditions (iv) and (vi) for $\chi_{A'} \ge 0$}
We have either $|A| = 3$ or $A = (2,2,2,2)$. The latter case together with $A = (2,2,2)$ are very special due to many symmetries and weak algebra structure. We investigate them in details in Section~\ref{section: examples}. The rest of this section focuses on the case $|A| = 3$, $A \neq (2,2,2)$.

In what follows we need to examine the structure of power series $p(\bt) \in \CC[[t_{i,j},q]]$ with $q = \exp(t_{\mu_A})$. For any monomial $\phi$ in variables $q$ and $t_{i,j}$ denote by $[\phi]p(\bt)$ the coefficient of the monomial $\phi$ in the series expansion of $p(\bt)$. 

We are going to make use of the following two propositions.
\begin{proposition}[{\cite[Proposition 3.15]{ist:1}} (see also Remark~3.14 in loc.cit.)]\label{prop: r=3 q coefficient}
Let $A$ be such that $1 \le a_1 \le a_2 \le a_3$ and $\F_{A}$ satisfies conditions (i),(ii),(iii),(v). We have the three cases.

\begin{description}
 \item[(a)] Let $a_1 \ge 3$. Then $[t_{i,\alpha}t_{j,\beta}t_{k,\gamma}q] \F_A$ is none-zero only if $\alpha=\beta=\gamma=1$ and $\{i,j,k\} = \{1,2,3\}$.
 \item[(b)] Let $A = (2,2,a_3)$ for some $a_3 \ge 3$ and $\F_{A}$ be symmetric in variables $t_{1,1},t_{2,1}$. Then $[t_{i,\alpha}t_{j,\beta}t_{k,\gamma}q] \F_A$ is none-zero only if $\alpha=\beta=\gamma=1$ and $\{i,j,k\} = \{1,2,3\}$.
 \item[(c)] Let $A = (1,2,2)$ and be $\F_{A}$ symmetric in variables $t_{2,1},t_{3,1}$. Then $[t_{i,\alpha}t_{j,\beta}q] \F_A$ is none-zero only if $\alpha=\beta=1$ and $\{i,j\} = \{2,3\}$.
\end{description}
\end{proposition}

\begin{proposition}[{\cite[Proposition 3.24]{ist:1}} (see also remark above)]\label{prop: r=3 noQ coefficient}
Let $A = (a_1,a_2,a_3) \neq (2,2,2)$ and $\F_{A}$ satisfy conditions (i),(ii),(iii),(v),(vi). 
If $a_i = a_j = 2$ assume in addition $\F_{A}$ to be symmetric in the variables $t_{i,1},t_{j,1}$.

Then $\F_A$ satisfies condition~(iv).
\end{proposition}

{\bf Assume $A$ being as in case~(a) above.} 
We have
\begin{align}
    \prod_{i=1}^3 \prod_{l=1}^{n_i} t_{(i,l),1} e^{t_{\mu_{A}}} = \prod_{i=1}^3 \left( v_{g_i^0,1} \right)^{n_i} e^{|G|v_{\mu_A}} + \dots.
\end{align}
Due to proposition above we have
\begin{align}
    \left[ \prod_{i=1}^3 \prod_{l=1}^{n_i} t_{(i,l),1} e^{t_{\mu_A}} \right] \F(\bt) = |G| \left[ \prod_{i=1}^3 \left( v_{g_i^0,1} \right)^{n_i} e^{|G|v_{\mu_A}}\right] \F_G.
\end{align}
By invariant secto axiom of $M_{\fAp,G}^\infty$ we have in the invariant sector that the coefficient of $v_{g_1^0,1}v_{g_2^0,1}v_{g_3^0,1}e^{|G|v_{\mu_A}}$ is equal to $1/|G|$ by the following lemma beneath. This gives condition (vi) for $\F(\bt)$.

\begin{lemma}\label{lemma: GW of the orbifolds covering}
  Let the pair $(\fAp,G)$ be as above.
  The term 
  \begin{equation}\label{eq:covering}
    \left(\prod_{i=1}^{3}(t_{i,|G/K_i|})^{n_i}\right)e^{|G|t_{\mu_{A'}}}
  \end{equation}
  occurs with the coefficient $1/|G|$ in $\F_{\fAp,\zeta^\infty}$. 
\end{lemma}
\begin{proof}
Consider the potential of the GW--theory of $\F_{\PP^1_{A'}}$. Due to the mirror symmetry Theorem~\ref{theorem: msUnorbifolded} it is enough to show the same statement about the potential $\F_{\PP^1_{A'}}$.

Consider the weight set $A$ associated to $A^\prime$ and $G$ as above.
The term \eqref{eq:covering} counts the Gromov--Witten invariants 
for the covering map $\PP^1_{A,\Lambda}\longrightarrow \PP^1_{A'}=[\PP^1_{A,\Lambda}/G]$.
Therefore, the coefficient is the inverse of the order of the covering transformation group, 
which is $1/|G|$.
\end{proof}
This allows us to apply Proposition~\ref{prop: r=3 noQ coefficient} giving us condition~(iv) too.
\\
\\
\indent We are going to use the same invariant sector argument for the remaining cases (b) and (c). For that we also need to show that additional symmetry conditions hold.

{\bf Assume $A$ being as in case~(b) above,} $A = (2,2,k)$ with $k \ge 3$. Such a set can only come from $A' = (2,2,2k)$ considered with $G = K_1$. Therefore $\F(\bt)$ is a function of $t_{1}$, $t_{5}$ and $t_{(1,1),1}$, $t_{(1,2),1}$, $t_{(3,1),1}$. The function $\F(\bt)$ is symmetric in $t_{(1,1),1}, t_{(1,2),1}$ due to $\Aut$--invariance axiom. Conditions of Proposition~\ref{prop: r=3 q coefficient} are fulfilled and we can apply the same invariant sector argument as above.
This gives conditions (vi) and (iv) for $\F(\bt)$.

{\bf Assume $A$ being as in case~(c) above,} $A = (1,2,2)$. Such a set can only come from $A' = (2,2,4)$ considered with $G = G^D$, $A' = (2,2,2)$ considered with $G = K_3$ or $A' = (1,2m,2m)$ with order $m$ group $G \subset K_1$. In the first two cases $\F(\bt)$ to satisfy the symmetry condition needed exactly in the same way as in case~(b) above. In the last case $\F(\bt)$ is a function of $t_{(2,1),1}$ and $t_{(3,1),1}$. It is symmetric in these variables because $\F(\bt)$ is fully defined by invariant sector axiom, where this symmetry condition holds. We conclude that conditions (vi) and (iv) hold for $\F(\bt)$.

\subsection{Conditions (iv) and (vi) for $\chi_{A'} < 0$}
We have $r := |A| \ge 4$ and $A \neq (2,2,2,2)$. We have the following analogue of Proposition~\ref{prop: r=3 q coefficient}.

\begin{proposition}[{\cite[Proposition 3.4]{shi:1}}]\label{prop: higher r q coefficient}
Let $A$ be such that $r:=|A| > 3$, $\chi_A < 0$ and $\F_{A}$ satisfies conditions (i),(ii),(iii),(v).
Then the coefficient $\left[\prod_{k=1}^r t_{i_k,\alpha_1}\dots t_{i_k,\alpha_{p_k}} \cdot e^{t_{\mu_A}} \right] \F_A$ with $\sum_{k=1}^r p_k \ge r$ is none-zero only if $\alpha_\bullet =1$ and $\{i_1,\dots,i_r\} = \{1,\dots,r\}$.
\end{proposition}

It follows that we have like in the section above
\begin{align}
    \left[ \prod_{i=1}^3 \prod_{l=1}^{n_i} t_{(i,l),1} e^{t_{\mu_A}} \right] \F = |G| \left[ \prod_{i=1}^r \left( v_{g_i^0,1} \right)^{n_i} e^{|G|v_{\mu_A}}\right] \F_G.
\end{align}
The right hand side coefficient is equal to $1/|G|$ by invariant sector axiom and Lemma~\ref{lemma: GW of the orbifolds covering}. 
This gives condition (vi) for $\F(\bt)$.

However we dont have an analogue of Proposition~\ref{prop: r=3 noQ coefficient} in \cite{shi:1}\footnote{one shoud note that it was present in some form in the preprint version of this paper}. In order to show condition (iv) we examine WDVV equation on $\F(\bt)$. The proof is divided into the following three cases:
\begin{itemize}
  \item[Case 1.] $G$ is arbitrary, there is index $i_0$ such that $a_{i_0} > 2$,
  \item[Case 2.] $A = (2,\dots,2)$ and $G \subseteq K_{i_0}$,
  \item[Case 3.] $A = (2,\dots,2)$ and $G = G^D = \langle(\epi[\frac{1}{2}],\epi[\frac{1}{2}],\epi[0]),(\epi[\frac{1}{2}],\epi[0],\epi[\frac{1}{2}]) \rangle$.
\end{itemize}

The difference between these cases is given by the algebra structure of $\overline{\A^*}(\fAp,G)$. It turns out to be very reach in the first case, however when $A$ is composed of $2$'s only this algebra structure turns out to be very simple encoding the pairing only. 

In what follows we use the \textit{correlators}:
  $$
    \langle t_{\alpha_1}, \dots, t_{\alpha_k} \rangle := \frac{\p^k \F(\bt)}{\p t_{\alpha_1} \dots \p t_{\alpha_k}}\mid_{\bt = \exp(t_{\mu_A}) = 0} \quad \in \CC.
  $$
Because the expression on the RHS is completely symmetric in $t_{\alpha_p}$ we will drop sometimes the commas on the LHS multiplying repeating indices.
We also use the following notation:
$$
  \wdvv(t_\alpha, t_\beta, t_\gamma, t_\delta) := \sum_{\sigma, \bar \sigma} \left(
  \frac{\p^3 \F}{\p t_\alpha \p t_\beta \p t_\sigma} \eta^{\sigma,\bar \sigma} \frac{\p^3 \F}{\p t_{\bar \sigma} \p t_\gamma \p t_\delta}
  - \frac{\p^3 \F}{\p t_\alpha \p t_\gamma \p t_\sigma} \eta^{\sigma,\bar \sigma} \frac{\p^3 \F}{\p t_{\bar \sigma} \p t_\beta \p t_\delta} \right).
$$

The following proposition holds true for all three cases.
\begin{proposition}\label{proposition: i--th sector of GW}
  Fix some numbers $p_1,p_2,p_3$ such that $1 \le p_i \le n_i$. Then for any positive $b_1$,$b_2$,$b_3$ such that at least two of them are non--zero, the function $\F(\bt) \mid_{t_1 = \exp(t_{\mu_A}) = 0}$ does not involve the term $t_{(1,p_1),j_1}^{b_1} t_{(2,p_2),j_2}^{b_2} t_{(3,p_3),j_3}^{b_3}$ for any indices $j_1,j_2,j_3$. 
  
  In correlators notations this reads:
  $$
    \left\langle t_{(1,p_1),j_1}^{b_1} t_{(2,p_2),j_2}^{b_2} t_{(3,p_3),j_3}^{b_3} \right\rangle = 0
  $$
\end{proposition}
\begin{proof}
 If $b_1=b_2=b_3=1$ this correlator is defined by structure constants of $\overline{\A^*}(\fAp,G)$. It follows from Proposition~\ref{prop: structure constants of xijk} and Eq.\eqref{eq: xik1 xik2} that in this case it vanishes.
 
Assume at least one of $b_1,b_2,b_3$ is greater than $1$.
  In the potential $\F_G(\bv)$ set all the untwisted sector variables $v_{g_i^a,j}$ to zero. 
  In the coordinates $\bt$ the equality $v_{g_i,j} = \dots = v_{g_i^a,j} = \dots = v_{g_i^{n_i-1},j} = 0$ is equivalent to $t_{(i,1),j} = \dots = t_{(i,k),j} = \dots = t_{(i,n_i),j}.$
  By condition (d) of $M_{\fAp,G}^\infty$ and Mirror symmetry Theorem~\ref{theorem: msUnorbifolded} the function obtained defines a submanifold of the orbifold GW--theory and hence satisfies Condition~(iv). Due to symmetry axiom the proof follows.
\end{proof}

In order to complete the proof of Condition~(iv) it remains to show that the mixed terms involving the variables $t_{(i,k),j}$ with the same index $i$ and different indices $k$ do not appear in $\F(\bt) \mid_{t_1=\exp(t_{\mu_A}) = 0}$. 

\subsubsection{Case 1: there is $i_0$ such that $a_{i_0} > 2$, the group $G$ is arbitrary}

\begin{proposition}\label{proposition: condition (iv) in v}
  Fix some $1 \le i_0 \le 3$ such that $a_{i_0} > 2$.
  The potential $\F(\bt)$ expansion does not contain a terms $t_{(i_0,k_1),j_1} \dots t_{(i_0,k_p),j_p}$ such that not all $k_j$ are equal.
\end{proposition}
\begin{proof}
  For this proof we adopt the notation $t_{k,j} := t_{(i_0,k),j}$.
  We prove the proposition by using the induction on the length $p$ of the term $T := t_{(i_0,k_1),j_1} \dots t_{(i_0,k_p),j_p} $. 
  
  {\bf Step 0: Assume $p = 3$.} Then the statement follows by Corollary~\ref{proposition: oJac structure}.

  {\bf Step 1: Assume $p = 4$.} Let $T = t_{k_1,j_1}t_{k_2,j_2}t_{k_3,j_3}t_{k_4,j_4}$.
  From the quasihomogeneity condition there is at least one index $j_\bullet > 1$. Let it be $j_3$. Consider the derivate w.r.t. $\p/\p t_\sigma$ of $\wdvv(t_{k_1,j_1},t_{k_2,j_2},t_{k_3,\kappa},t_{k_3,j_3-\kappa})$ for some $\kappa < j_3$
  \begin{align*}
    &\sum_{\gamma,\bar\gamma} \Bigg( 
    \langle t_{k_1,j_1},t_{k_2,j_2},t_\sigma,t_\gamma \rangle \eta^{\gamma,\bar \gamma} \langle t_{\bar\gamma},t_{k_3,\kappa},t_{j_3-\kappa}\rangle 
    + \langle t_{k_1,j_1},t_{k_2,j_2},t_\gamma \rangle \eta^{\gamma,\bar \gamma} \langle t_{\bar\gamma},t_\sigma,t_{k_3,\kappa},t_{j_3-\kappa}\rangle \Bigg)
    \\
    &\quad =
    \sum_{\gamma,\bar\gamma} \Bigg(
    \langle t_{k_1,j_1},t_{k_3,\kappa},t_\sigma,t_\gamma \rangle \eta^{\gamma,\bar \gamma} \langle t_{\bar\gamma},t_{k_2,j_2},t_{k_3,j_3-\kappa}\rangle 
    + \langle t_{k_1,j_1},t_{k_3,\kappa},t_\gamma \rangle \eta^{\gamma,\bar \gamma} \langle t_{\bar\gamma},t_\sigma,t_{k_2,j_2},t_{k_3,j_3-\kappa}\rangle \Bigg)
    \\
    \Leftrightarrow
    \langle &t_{k_1,j_1},t_{k_2,j_2},t_\sigma,t_{k_3,j_3} \rangle  
    + \sum_{\gamma,\bar\gamma} \langle t_{k_1,j_1},t_{k_2,j_2},t_\gamma \rangle \eta^{\gamma,\bar \gamma} \langle t_{\bar\gamma},t_\sigma,t_{k_3,\kappa},t_{k_3,j_3-\kappa}\rangle
    \\
    &=
    \sum_{\gamma,\bar\gamma} \Bigg( \langle t_{k_1,j_1},t_{k_3,\kappa},t_\sigma,t_\gamma \rangle \eta^{\gamma,\bar \gamma} \langle t_{\bar\gamma},t_{k_2,j_2},t_{k_3,j_3-\kappa}\rangle 
    + \langle t_{k_1,j_1},t_{k_3,\kappa},t_\gamma \rangle \eta^{\gamma,\bar \gamma} \langle t_{\bar\gamma},t_\sigma,t_{k_2,j_2},t_{k_3,j_3-\kappa}\rangle \Bigg).
  \end{align*}
  If $k_1,k_2,k_3$ are pairwise distinct the last computation shows vanishing of the four-point correlator in question by taking $\sigma = (k_4,j_4)$. 

  We should consider now two more cases: case $1$ is when $k_1 = k_2 \neq k_3 = k_4$ and case $2$ when $k_1=k_2=k_3\neq k_4$. We have $2 a_{i_0} = j_1+j_2+j_3+j_4$ and both $j_1 + j_2$, $j_3+j_4$ can not be smaller than $a_{i_0}$. Assume $j_1+j_2 \ge a_{i_0}$. The WDVV expression above gives:
  \begin{align*}
    \langle t_{k_1,j_1},t_{k_2,j_2},t_\sigma,t_{k_3,j_3} \rangle  
    &= - \delta_{j_1+j_2 < a_{i_0}}\langle t_{k_1,j_1},t_{k_1,j_2},t_{k_1,a_{i_0}-j_1-j_2} \rangle \frac{1}{a_{i_0}} \langle t_{k_1,j_1+j_2},t_\sigma,t_{k_3,\kappa},t_{k_3,j_3-\kappa}\rangle
    \\
    & = - \delta_{j_1+j_2 < a_{i_0}} \cdot \langle t_{k_1,j_1+j_2},t_\sigma,t_{k_3,\kappa},t_{k_3,j_3-\kappa}\rangle = 0.
  \end{align*}
  It completes the proof for $p=4$ because we may put any $\sigma$ needed.
  
  {\bf Step 2: Assume $p > 4$.} We proceed by induction in $p$.
  Let the proposition be proved for all $p \le l$. We show it for $p = l+1$. Due to quasihomogeneity we can assume again $j_3 > 1$. Fix some $1 \le \kappa < j_3$ and set $I = \lbrace (k_4,j_4),\dots, (k_{l+1},j_{l+1})\rbrace$ of $\wdvv(t_{k_1,j_1},t_{k_2,j_2},t_{k_3,\kappa},t_{k_3,j_3-\kappa})$.
  
  Consider the derivate w.r.t. $\p^{l-2}/\p t_{k_4,j_4}\dots \p t_{k_{l+1},j_{l+1}}$ of $\wdvv(t_{k_1,j_1},t_{k_2,j_2},t_{k_3,\kappa},t_{k_3,j_3-\kappa})$
  \begin{align*}
    \sum_{\gamma,\bar\gamma} & \Bigg( \langle t_{k_1,j_1},t_{k_2,j_2},\bt_{I}, t_\gamma \rangle \eta^{\gamma,\bar \gamma} \langle t_{\bar\gamma},t_{k_3,\kappa},t_{k_3,j_3-\kappa}\rangle 
    + \langle t_{k_1,j_1},t_{k_2,j_2}, t_\gamma \rangle \eta^{\gamma,\bar \gamma} \langle t_{\bar\gamma},\bt_I,t_{k_3,\kappa},t_{k_3,j_3-\kappa}\rangle
    \\
    &\quad
    - \langle t_{k_1,j_1},t_{k_3,\kappa},\bt_I,t_\gamma \rangle \eta^{\gamma,\bar \gamma} \langle t_{\bar\gamma},t_{k_2,j_2},t_{k_3,j_3-\kappa}\rangle 
    - \langle t_{k_1,j_1},t_{k_3,\kappa},t_\gamma \rangle \eta^{\gamma,\bar \gamma} \langle t_{\bar\gamma},\bt_I,t_{k_2,j_2},t_{k_3,j_3-\kappa}\rangle \Bigg)
    \\
    =& 
    \sum_{\gamma,\bar\gamma} \sum_{\substack{I_A \sqcup I_A = I, \\ |I_A| \neq 0, |I_B| \neq 0}} 
     \Bigg( \langle t_{k_1,j_1},t_{k_3,\kappa},\bt_{I_A},t_\gamma \rangle \eta^{\gamma,\bar \gamma} \langle t_{\bar\gamma}, \bt_{I_B}, t_{k_2,j_2},t_{k_3,j_3-\kappa}\rangle 
    \\
     &\quad\quad -\langle t_{k_1,j_1},t_{k_2,j_2},\bt_{I_A}, t_\gamma \rangle \eta^{\gamma,\bar \gamma} \langle t_{\bar\gamma},\bt_{I_B},t_{k_3,\kappa},t_{k_3,j_3-\kappa}\rangle \Bigg).
  \end{align*}
  RHS vanishes by induction assumption and we get
  \begin{align*}
    \langle &t_{k_1,j_1},t_{k_2,j_2},\bt_I,t_{k_3,j_3} \rangle = 
    \sum_{\gamma,\bar\gamma}\Bigg( - \langle t_{k_1,j_1},t_{k_2,j_2}, t_\gamma \rangle \eta^{\gamma,\bar \gamma} \langle t_{\bar\gamma},\bt_I,t_{k_3,\kappa},t_{k_3,j_3-\kappa}\rangle
    \\
    &
    + \langle t_{k_1,j_1},t_{k_3,\kappa},\bt_I,t_\gamma \rangle \eta^{\gamma,\bar \gamma} \langle t_{\bar\gamma},t_{k_2,j_2},t_{k_3,j_3-\kappa}\rangle 
    + \langle t_{k_1,j_1},t_{k_3,\kappa},t_\gamma \rangle \eta^{\gamma,\bar \gamma} \langle t_{\bar\gamma},\bt_I,t_{k_2,j_2},t_{k_3,j_3-\kappa}\rangle \Bigg).
  \end{align*}
  If $k_1,k_2,k_3$ are pairwise different this gives the vanishing of the correlator needed by Step~0 above. If there is no triple of pairwide different indices $k_\alpha$, $k_\beta$, $k_\gamma$ the RHS in the expression above vanishes due to quasihomogeneity and Step~0.
\end{proof}

\subsubsection{Case 2: $A = (2,\dots,2)$ and $G \subseteq K_l$}
For all conditions of Theorem~\ref{thm:satisfies ISTr} expect condition (iv), proof of this case does not differ from the previous one. 
We skip it here because it is completely analogous. However we make use of a slightly different change of the variables.

Without loss of generality we can assume $l=3$. Let $n := n_l$. There are to possibilities: $A' = (2n,2n,2)$ with $G \subset K_3$ or order $n$ and $A' = (n,n,2)$ with $G = K_3$. We only show the proof for the first one becasue the proof of the second one is completely parallel.

Assume $A' = (2n,2n,2)$ with $G \subset K_3$ or order $n$.
We have
\begin{align}
    \F(\bt) &= \frac{1}{2}t_{1}^2 t_{\mu_A} + \frac{t_1}{4} \left( t_{1,1}^2 + t_{2,1}^2 + \sum_{k = 1}^{n} t_{3,n}^2 \right)
    + \sum t_{1,1}^{\alpha_1} t_{2,1}^{\alpha_2} t_{3,1}^{\beta_1} \dots t_{3,n}^{\beta_n} g_{\alpha,\beta}(t_{\mu_A})
\end{align}
for some $g_{\alpha,\beta}$. From symmetry axiom we have $g_{\alpha,\beta}(t_{\mu_A}) = g_{\alpha,\beta'}(t_{\mu_A})$ if $\beta'$ is obtained from $\beta$ by a permutation. Let 
\[
    g_{\alpha,\beta} = \sum_{k \ge 0} c_{\alpha,\beta; k} q^k, \quad q = \exp(t_{\mu_A}).
\]
From the quasihomogeneity condtion on $\F(\bt)$ we have $c_{\alpha,\beta;0} = 0$ unless $\sum_i \alpha_i + \sum_i \beta_i = 4$. By using the same argument as above we have $c_{\alpha,\beta; 1} \neq 0$ with $\sum_i \alpha_i + \sum_i \beta_i = n+2$ only if $\alpha_1 = \alpha_2=1$ and $\beta_1=\dots=\beta_n=1$. When this condition is satisfied, the coefficient in question is equal to $1$. It remains to show that condition (iv) holds to get Theorem~\ref{theorem: main}.

By invariant sector axiom (see also Corollary~3.8 in \cite{shi:1}) and Proposition~\ref{proposition: condition (iv) in v} we have for some constants $c_4$, $c_{2,2}$, $c_{1,3}$ and $c_{1,1,2}$
\begin{align}
    \F(\bt) &= \frac{1}{2}t_{1}^2 t_{\mu_A} + \frac{t_1}{4} \left( t_{1,1}^2 + t_{2,1}^2 + \sum_{k = 1}^{n} t_{3,n}^2 \right)
    - \frac{1}{96} \left( t_{1,1}^4 + t_{2,1}^4\right) + \frac{c_4}{4!} \sum_{k=1}^n t_{3,k}^4
    \\
    &+ \frac{1}{4} c_{2,2} \sum_{i \neq j} t_{3,i}^2 t_{3,j}^2 + \frac{1}{6} c_{1,3} \sum_{i \neq j} t_{3,i} t_{3,j}^3 + \frac{1}{2} c_{1,1,2} \sum_{i \neq j \neq k} t_{3,i} t_{3,j} t_{3,k}^2 + O(q).
\end{align}

Consider $\wdvv(t_{1,1},t_{2,1},t_{3,1},t_{3,k})$ with $k \neq 1$. Derivating it w.r.t. $t_{3,1}\dots t_{3,n}$ and assuming the coefficient of $q$ we get
\begin{align}
    0 &= \sum_{p=1}^n \langle t_{1,1},t_{2,1},t_{3,1},\dots \widehat{p},\dots, t_{3,n}, t_x \rangle \eta^{x,\overline{x}} \langle t_{\overline{x}}, t_{3,p}, t_{3,1}, t_{3,k} \rangle
    \\
    &= \langle t_{1,1},t_{2,1},t_{3,1},\dots, t_{3,n} \rangle \cdot 2 \cdot \sum_{p=1}^n \langle t_{3,p}, t_{3,p}, t_{3,1}, t_{3,k} \rangle
\end{align}
what gives $(n-2)c_{1,1,2} + 2 c_{1,3} = 0$. Derivating the same WDVV w.r.t. Derivating it w.r.t. $t_{3,2},t_{3,2},t_{3,3}\dots t_{3,n}$ and assuming the coefficient of $q$ we get
\begin{align}
    0 &= \langle t_{1,1},t_{2,1},t_{3,2},\dots t_{3,n}, t_x \rangle \eta^{x,\overline{x}} \langle t_{\overline{x}}, t_{3,2}, t_{3,1}, t_{3,2} \rangle
    \\
    &= \langle t_{1,1},t_{2,1},t_{3,1},\dots, t_{3,n} \rangle \cdot 2 \cdot \langle t_{3,1}, t_{3,1}, t_{3,2}, t_{3,2} \rangle
\end{align}
giving us $c_{2,2} = 0$.

Consider $\wdvv(t_{1,1},t_{2,1},t_{3,1},t_{3,1})$. Derivating it w.r.t. $t_{3,2},t_{3,2},t_{3,3}\dots t_{3,n}$ and assuming the coefficient of $q$ we get
\begin{align}
    0 &= \langle t_{1,1},t_{2,1},t_{3,2},\dots , t_{3,n}, t_x \rangle \eta^{x,\overline{x}} \langle t_{\overline{x}}, t_{3,1}, t_{3,1}, t_{3,2} \rangle
    \\
    &= \langle t_{1,1},t_{2,1},t_{3,1}t_{3,2},\dots , t_{3,n}\rangle \cdot 2 \cdot \langle t_{3,1}, t_{3,1}, t_{3,1}, t_{3,2} \rangle
\end{align}
what gives $c_{1,3} = 0$.
Derivating the same WDVV expression w.r.t. $t_{3,1}\dots t_{3,n}$ we get
\begin{align}
    0 &= \langle t_{1,1},t_{2,1},t_{3,1},\dots , t_{3,n}, t_x \rangle \eta^{x,\overline{x}} \langle t_{\overline{x}}, t_{3,1}, t_{3,1}\rangle 
    \\
    &+ \sum_{p=1}^n \langle t_{1,1},t_{2,1},t_{3,1},\dots \widehat{p} \dots , t_{3,n}, t_x \rangle \eta^{x,\overline{x}} \langle t_{\overline{x}}, t_{3,p}, t_{3,1}, t_{3,1} \rangle
\end{align}
what gives $\frac{1}{2} + 2 c_{0,4} + 2(n-1) c_{2,2} = 0$.

Assuming all the equations written we get $c_{0,4} = -1/4$ and $c_{1,1,2} = c_{2,2} = c_{1,3} = 0$ and the potential obtined satisfies condtion (iv) of Theorem~\ref{theorem: yuuki--GW}. This completes the proof of Theorem~\ref{theorem: main} in this case.

\subsubsection{Case 3: $A = (2,\dots,2)$ and $G = G^D$}
In this case we could only have $A = (2,2,2,2,2,2)$ and $\fAp = x_1^4 + x_2^4 + x_3^4 - s_{\mu}^{-1} x_1x_2x_3$.
Let $g$ and $h$ be the generators of $G^D$. Then we have $G^D = \{1,g,h,gh\}$. Writing the potential $\F_G(\bv)$ in coordinates we get a rather particular expression by using the axioms we introduce.

The coordinate change assumes simple form
\begin{align}
    t_{(1,1),1} &= v_{g^0,1} + v_{g,1}, \ t_{(1,2),1} = v_{g^0,1} - v_{g,1},
    \\
    t_{(2,1),1} &= v_{h^0,1} + v_{h,1}, \ t_{(2,2),1} = v_{h^0,1} - v_{h,1},
    \\
    t_{(3,1),1} &= v_{(gh)^0,1} + v_{gh,1}, \ t_{(3,2),1} = v_{(gh)^0,1} - v_{gh,1}.
\end{align}

It follows from axiom~(b), symmetry axiom and Proposition~\ref{proposition: i--th sector of GW} that the potential $\F_G(\bt)$ has the form
\begin{align}
    \F_G(\bt) &= \frac{1}{2} t_1^2t_{8} + \frac{1}{4} \sum_{i=1}^3 \left( t_{(i,1),1}^2 + t_{(i,2),1}^2 \right) + \sum_{i=1}^3 c_{1,i}\left( t_{(i,1),1}^4 + t_{(i,2),1}^4 \right)
    \\
    & + \sum_{i=1}^3 \left( \frac{c_{2,i}}{6} ( t_{(i,1),1} t_{(i,2),1}^3 + t_{(i,1),1}^3 t_{(i,2),1}) + \frac{c_{3,i}}{4} t_{(i,1),1}^2 t_{(i,2),1}^2  \right) + q \cdot \phi(\bt) + O(q^2),
\end{align}
where $q = \exp(t_{\mu})$. One gets from invariant sector axiom that $c_{1,i} = - 1 / 96$. From Proposition~\ref{prop: higher r q coefficient} we have $\phi(\bt) = \prod_{i=1}^3\prod_{l=1}^2 t_{(i,l),1}$. Taking the coefficient of $q$ in $\wdvv(t_{(1,1),1},t_{(1,1),1},t_{(2,1),1},t_{(2,2),1})$ we get
\[
 2 c_{2,1} t_{(3,1),1} t_{(3,2),1} t_{(1,1),1}^2+\frac{8}{3} c_{3,1} t_{(1,2),1} t_{(3,1),1} t_{(3,2),1} t_{(1,1),1}+2 c_{2,1} t_{(1,2),1}^2 t_{(3,1),1} t_{(3,2),1} = 0
\]
that should hold in polynomial ring and therefore we have both $c_{2,1} = 0$ and $c_{3,1} = 0$. 

\section{Examples}\label{section: examples}
In this section we prove case by case Theorem~\ref{theorem: main} for the pairs $(\fAp,G)$, s.t. $A=(2,2,2)$ or $A = (2,2,2,2)$. 
Conditions (i),(ii),(iii) and (v) hold for these pairs too by Section~\ref{section: MS proof part 1} and we focus on conditions (iv) and (vi).

Our strategy is the following. First write the potential $\F_{(\fAp,G)}(\bt)$ in the flat coodinates introduced in Eq.\eqref{eq: change of variables t=t(v)}. Quasihomogeneity and expansion axioms assure that the potential can be expanded in a series in $\bt$ and $\exp(t_{\mu_{A}})$. Some of the coefficients of this series expansion can further be identified by $\Aut$--invariance and $G$--grading axioms. 

Next we compute the functions $\F_\fAp^G$ and $\F_{(\fAp,G)}^\id$, first being the restriction of $\F_{\fAp}$ to the $G$--invariant locus and second the restriction of $\F_{(\fAp,G)}$ to identity sector. By identity sector axiom we have $\F_\fAp^G(\bt') = |G|\F_{(\fAp,G)}^\id(\bt')$ for some variables $\bt'$.

Finally we solve WDVV equation.

\subsection{The weight set $A = (2,2,2)$}
We should consider $A' = (2,3,3)$ with $G = K_1$ and $A' = (2,2,4)$ with $G = K_1$. For all three cases the following holds.

Due to quasihomogeneity and expansion axioms we see that after some reindexing of the variables $\F_{\fAp,G}(\bt)$ is a polynomial in $t_1,t_{1,1},t_{2,1},t_{3,1}$ and $t_{\mu_A}=t_5$. We have for some complex numbers $c_\alpha = c_{\alpha_1,\alpha_2,\alpha_3,\alpha_4}$
\begin{align}
    \F_{\fAp,G}(\bt) &= \frac{1}{2} t_{5} t_1^2 + \frac{1}{4} t_1 \left(t_{1,1}^2+t_{2,1}^2+t_{3,1}^2\right) 
    + \sum c_{\alpha} t_{1,1}^{\alpha_1} t_{2,1}^{\alpha_2} t_{3,1}^{\alpha_3} e^{\alpha_4 t_5},
\end{align}
the summation being taken over all non--negative integers $\alpha_\bullet$, such that ${\alpha_1+\alpha_2+\alpha_3+\alpha_4 = 4}$. 
From the algebra structure at the origin we know that $c_{\alpha} = 0$ if $\alpha_4 =0$ and at least two of $\alpha_1,\alpha_2,\alpha_3$ are non--zero. 

\subsubsection{$A' = (2,3,3)$ with $G = K_1$}
Due to the $\Aut$--invariance axiom we have $c_{\alpha_1,\alpha_2,\alpha_3,\alpha_4} = c_{\alpha_2,\alpha_1,\alpha_3,\alpha_4} = c_{\alpha_3,\alpha_2,\alpha_1,\alpha_4}$.
We write $\F_{\fAp,G}$ in the coordinates $t_{1,1} = t_{(1,1),1}$, $t_{2,1} = t_{(1,2),1}$, $t_{3,1} = t_{(1,3),1}$.

By using mirror theorem \ref{theorem: msUnorbifolded} we have $\F_{A'} = \F_{\PP^1_{A'}}$. The latter potential can be found in \cite{rossi}. 
It is a function of $t_1,t_5$ and $t_{1,1}$, $t_{2,1}$, $t_{2,2}$, $t_{3,1}$, $t_{3,2}$.
Restriction of $\F_{A'}$ to the fixed locus of the group action and also restriction to the $\id$--sector of $\F_{\fAp,G}$ read
\begin{align}
    \F_{A'}^G &= -\frac{t_{1,1}^4}{96}+\frac{1}{3} e^{3 t_5} t_{1,1}^3+\frac{1}{2} e^{6 t_5} t_{1,1}^2+\frac{1}{4} t_{1,1} t_{1,1}^2+\frac{1}{12} e^{12 t_5}+\frac{1}{2} t_5 t_{1,1}^2
    \\
    \F_{\fAp,G}^\id &= 3 t_{1,1}^4 c_{0,0,4,0}+3 e^{t_5} t_{1,1}^3 c_{0,0,3,1}+e^{t_5} t_{1,1}^3 c_{1,1,1,1}+6 e^{t_5} t_{1,1}^3 c_{2,0,1,1}
    \\
    & +3 e^{2 t_5} t_{1,1}^2 c_{0,0,2,2}+3 e^{2 t_5} t_{1,1}^2 c_{1,0,1,2}+3 e^{3 t_5} t_{1,1} c_{0,0,1,3}+e^{4 t_5} c_{0,0,0,4}
    \\
    &+\frac{3}{4} t_1 t_{1,1}^2+\frac{1}{2} t_5 t_1^2.
\end{align}
Equating $\F_{\fAp,G}^\id(t_1,t_{1,1},t_5) = 3 \cdot \F_{A'}^G (t_1,t_{1,1},t_5/3)$ we have
\begin{align}
    & c_{0,0,0,4} = \frac{1}{4}, \ c_{0,0,1,3}= 0, \ c_{0,0,4,0} = -\frac{1}{96}, 
    \\
    & c_{1,0,1,2} = \frac{1}{2}-c_{0,0,2,2}, \ c_{2,0,1,1} = -\frac{1}{2} c_{0,0,3,1} - \frac{1}{6} c_{1,1,1,1}+\frac{1}{6}.
\end{align}
Finally the only unknown coefficients in $\F_{\fAp,G}$ remain $c_{0,0,2,2},c_{0,0,3,1},c_{1,1,1,1}$. They are resolved by WDVV equation to be
\[
    c_{0,0,2,2} = \frac{1}{2}, \ c_{0,0,3,1} = 0,c_{1,1,1,1} = 1,
\]
giving us exactly the potential of $\F_{\PP^1_{2,2,2}}$. 

In the $\bv$ coordinates the potential reads
\begin{align*}
 \F_{\fAp,K_1}(\bv) &= \frac{q^{12}}{12} + \left(v_{g^2,1} v_{g^3,1}+\frac{v_{g,1}^2}{2}\right) q^{6}+\frac{1}{3} \left(v_{g^3,1}^3+v_{g^2,1}^3-3 v_{g^2,1} v_{g^3,1} v_{g,1}+v_{g,1}^3\right) q^{3}
 \\
 & -\frac{1}{96} \left(12 v_{g^2,1} v_{g^3,1} v_{g,1}^2 + 4 v_{g^3,1}^3v_{g,1} + 4v_{g^2,1}^3v_{g,1} + 6 v_{g^2,1}^2 v_{g^3,1}^2 + v_{g,1}^4\right)
 \\
 &+\frac{1}{4} \left(2 v_{g^2,1} v_{g^3,1}+v_{g,1}^2\right) v_{\text{id},1} +\frac{1}{2} v_{\text{id},1}^2 v_{\text{id},5}
\end{align*}
with $q = \exp(v_{\text{id},5})$.

\subsubsection{$A' = (2,2,4)$ with $G = K_1$}
Due to the symmetry $\Aut$--invariance axiom we have $c_{\alpha_1,\alpha_2,\alpha_3,\alpha_4} = c_{\alpha_2,\alpha_1,\alpha_3,\alpha_4}$.
We write $\F_{\fAp,G}$ in the coordinates $t_{1,1} = t_{(1,1),1}$, $t_{2,1} = t_{(1,2),1}$, $t_{3,1} = t_{(3,1),1}$.

By using mirror theorem \ref{theorem: msUnorbifolded} we have $\F_{A'} = \F_{\PP^1_{A'}}$. The latter potential can be found in \cite{rossi}. It is a function of $t_1, t_{1,1}, t_{2,1},t_{3,1},t_{3,2}, t_{3,3}$. Restriction of $\F_{A'}$ to the fixed locus of the group action and also restriction to the $\id$--sector of $\F_{\fAp,G}$ read
\begin{align}
    \F_{A'}^G &= \frac{e^{8 t_5}}{8} +\frac{e^{4 t_5}}{4} ( 2t_{1,1}^2 + t_{3,2}^2) + \frac{e^{2 t_5}}{2} t_{3,2} t_{1,1}^2 
    -\frac{t_{1,1}^4}{96} -\frac{t_{3,2}^4}{192} +\frac{t_5 t_1^2}{2} +\frac{t_1}{8} (2t_{1,1}^2 + t_{3,2}^2)
    \\
    \F_{\fAp,G}^\id &= 2 t_{1,1}^4 c_{0,4,0,0} +2 e^{t_5} t_{1,1}^3 c_{0,3,0,1}+2 e^{t_5} t_{1,1}^3 c_{1,2,0,1} +2 e^{2 t_5} t_{1,1}^2 c_{0,2,0,2} 
    \\
    &+e^{2 t_5} t_{1,1}^2 c_{1,1,0,2}+2 e^{t_5} t_{3,2} t_{1,1}^2 c_{0,2,1,1} +e^{t_5} t_{3,2} t_{1,1}^2 c_{1,1,1,1} +2 e^{t_5} t_{3,2}^2 t_{1,1} c_{0,1,2,1}
    \\
    & +2 e^{3 t_5} t_{1,1} c_{0,1,0,3} +2 e^{2 t_5} t_{3,2} t_{1,1} c_{0,1,1,2}+t_{3,2}^4 c_{0,0,4,0}+e^{t_5} t_{3,2}^3 c_{0,0,3,1}
    \\
    &+e^{2 t_5} t_{3,2}^2 c_{0,0,2,2}+e^{4 t_5} c_{0,0,0,4}+e^{3 t_5} t_{3,2} c_{0,0,1,3}+\frac{1}{2} t_5 t_1^2+\frac{1}{4} t_1 \left(2 t_{1,1}^2+t_{3,2}^2\right).
\end{align}
Equating $\F_{\fAp,G}^\id(t_1,t_{1,1},t_{3,2},t_5) = 2 \cdot \F_{A'}^G (t_1,t_{1,1},t_{3,2},t_5/2)$ we get
\begin{align}
    & c_{0,0,0,4} =  \frac{1}{4}, \ c_{0,0,1,3} = 0, \ c_{0,0,2,2} = \frac{1}{2}, \ c_{0,0,3,1} = 0, \ c_{0,0,4,0} = -\frac{1}{96},
    \\
    & c_{0,1,0,3} = 0, \ c_{0,1,1,2} = 0, \ c_{0,1,2,1} = 0, \ c_{0,4,0,0} = -\frac{1}{96}, \ c_{1,1,0,2} = 1-2 c_{0,2,0,2},
    \\
    & c_{1,1,1,1} = 1-2 c_{0,2,1,1}, \ c_{1,2,0,1} = -c_{0,3,0,1}.
\end{align}
The remaining unknown coefficients are resolved from WDVV equation giving
\[
    c_{0,2,0,2} = \frac{1}{2}, \ c_{0,2,1,1} = 0, \ c_{0,3,0,1} = 0,
\]
so that $\F_{\fAp,G} = \F_{\PP^1_A}$.

In the $\bv$ coordinates the potential reads
\begin{align*}
 \F_{\fAp,K_1}(\bv) &= \frac{q^8}{8} + \left(\frac{1}{2} v_{g^2,1}^2+\frac{v_{g,1}^2}{2}+\frac{v_{\text{id},x}^2}{4}\right)q^4 
 + \left(\frac{1}{2} v_{g,1}^2 v_{\text{id},x}-\frac{1}{2} v_{g^2,1}^2 v_{\text{id},x}\right)q^2 
 \\
 & -\frac{1}{96} v_{g^2,1}^4-\frac{1}{16} v_{g,1}^2 v_{g^2,1}^2-\frac{v_{g,1}^4}{96}  -\frac{v_{\text{id},x}^4}{192}
 \\
  &+v_{\text{id},1} \left(\frac{1}{4} v_{g^2,1}^2+\frac{v_{g,1}^2}{4}+\frac{v_{\text{id},x}^2}{8}\right)
 +\frac{1}{2} v_{\text{id},1}^2 v_{\text{id},5}
\end{align*}
with $q = \exp(v_{\text{id},5})$.

\subsection{The weight set $A = (2,2,2,2)$}
We should consider $A' = (4,4,2)$ with $G = G^D$, $A' = (4,4,2)$ with $G = K_3$ and $A' = (6,3,2)$ with $G = K_3$. For all three cases the following holds.

Due to quasihomogeneity and expansion axioms after the suitable reindexing of the variables, $\F_{\fAp,G}(\bt)$ is a polynomial in $t_1$, $t_{1,1}$, $t_{2,1}$, $t_{3,1}$, $t_{4,1}$ and power series in $e^{t_6}$. 
We have for some functions $g_\alpha = g_{\alpha_1,\alpha_2,\alpha_3,\alpha_4}(t_{6})$
\begin{align}
    \F_{\fAp,G}(\bt) &= \frac{1}{2} t_6 t_1^2 + \frac{1}{4} t_1 \sum_{k=1}^4 t_{k,1}^2 
    + \sum t_{1,1}^{\alpha_1} t_{2,1}^{\alpha_2} t_{3,1}^{\alpha_3} t_{4,1}^{\alpha_4} g_{\alpha}(t_{6}),
\end{align}
the summation being taken over all non--negative integers $\alpha_\bullet$, such that ${\alpha_1+\alpha_2+\alpha_3+\alpha_4 = 4}$. 
The algebra structure at the origin is very simple, in this case it doesn't give any requirements on $g_\alpha(t_6)$ for any $\alpha$. 

For the later use consider also the functions
\begin{align}
 &\eta (q)=q^{1/24} \prod _{k=1}^{\infty} \left(1-q^k\right), \quad E_2(q)=1-24 \sum _{k=1}^{\infty } \frac{k q^k}{1-q^k},
 \\
 &\theta _2(q)=2 \sum _{k=0}^{\infty } q^{\frac{1}{2} \left(k+\frac{1}{2}\right)^2}, \quad \theta _3(q)=1 + 2 \sum _{k=1}^{\infty } q^{\frac{k^2}{2}}, 
 \\
 &\quad\quad \theta _4(q)=1 + 2 \sum _{k=1}^{\infty } (-1)^k q^{\frac{k^2}{2}}.
\end{align}
We will use them to express Gromov--Witten Frobenius manifold potentials of elliptic orbifolds. We will also write $\eta(t) = \eta(q)$, $E_2(t) = E_2(q)$ and $\theta_k(t) = \theta_k(q)$
being obtained by the change of the variables $q = \exp(t)$.

\subsubsection{$A' = (4,4,2)$ with $G = G^D$}\label{section: elliptic E_7 G^D case}
Write $\F_{\fAp,G}(\bt)$ in the coordinates $t_{1,1} = t_{(1,1),1}$, $t_{2,1} = t_{(1,2),1}$, $t_{3,1} = t_{(2,1),1}$, $t_{4,1} = t_{(2,2),1}$.

Due to $\Aut$--invariance axiom we have $g_{\alpha_1,\alpha_2,\alpha_3,\alpha_4} = g_{\alpha_3,\alpha_4,\alpha_1,\alpha_2}$.
By using mirror theorem \ref{theorem: msUnorbifolded} we have $\F_{A'} = \F_{\PP^1_{A'}}$. The latter potential can be found in \cite{BP18,B15} (see also \cite{SZ17} for the particular correlators). 
In this section denote
\begin{align}
    & x(t) = \left( \theta_3(4t) \right)^2, \ y(t) = \left(\theta_2(8 t) \right)^2,
    \\
    & w(t) = \frac{1}{3} \left( E_2(4t) - 2 E_2(8t) + 4 E_2(16t) \right).
\end{align}
We have for $q = \exp(t)$
\begin{align}
    x(q) &= 1+4 q^4+4 q^8+4 q^{16}+8 q^{20}+ +O\left(q^{27}\right),
    \\
    y(q) &= 4 q^2+8 q^{10}+4 q^{18}+8 q^{26}+ +O\left(q^{27}\right),
    \\
    w(q) &= 1-8 q^4-8 q^8-32 q^{12}-40 q^{16}-48 q^{20}-32 q^{24} +O\left(q^{27}\right).
\end{align}

Restriction of $\F_{A'}$ to the fixed locus of the group action is given by setting $t_{1,1} = t_{1,3} =0$, $t_{2,1} = t_{2,3} =0$ and $t_{3,1} = 0$. Denote also $t_{1,2} = t_2$ and $t_{2,2} = t_3$.

Restriction to the $\id$--sector of $\F_{\fAp,G}$ is then given by setting $t_{1,1} = t_{2,1} = t_1$ and $t_{3,1} = t_{4,1} = t_2$. We have
\begin{align}
    \F_{A'}^G &= \frac{1}{384} (t_2^4 + t_3^4) \left(-3 w\left(t_6\right)+x\left(t_6\right){}^2-2 y\left(t_6\right){}^2\right) 
    \\
    &+\frac{1}{64} t_3^2 t_2^2 \left(x\left(t_6\right){}^2-w\left(t_6\right)\right)+\frac{1}{2} t_6 t_1^2+\left(\frac{t_2^2}{8}+\frac{t_3^2}{8}\right) t_1,
    \\
    \F_{\fAp,G}^\id &= t_2^4 g_{0,0,2,2} \left(t_6\right) +2 t_2^4 g_{0,4,0,0}\left(t_6\right) +2 t_2^4 g_{1,3,0,0}\left(t_6\right) +4 t_3 t_2^3 g_{0,3,1,0}\left(t_6\right)
    \\
    &+4 t_3 t_2^3 g_{1,2,1,0}\left(t_6\right)+4 t_3^2 t_2^2 g_{0,2,1,1}\left(t_6\right)+4 t_3^2 t_2^2 g_{0,2,2,0}\left(t_6\right)+t_3^2 t_2^2 g_{1,1,1,1}\left(t_6\right)
    \\
    &+4 t_3^3 t_2 g_{0,3,1,0}\left(t_6\right)+4 t_3^3 t_2 g_{1,2,1,0}\left(t_6\right)+t_3^4 g_{0,0,2,2}\left(t_6\right)+2 t_3^4 g_{0,4,0,0}\left(t_6\right)
    \\
    &+2 t_3^4 g_{1,3,0,0}\left(t_6\right)+\frac{1}{2} t_6 t_1^2+\frac{1}{4} t_1 \left(2 t_2^2+2 t_3^2\right)
\end{align}
Equating $\F_{\fAp,G}^\id(t_1,t_2,t_3,t_6) = 4 \cdot \F_{A'}^G (t_1,t_2,t_3,t_6/4)$ we get
\begin{align*}
    & g_{1,1,1,1}\left(t_6\right) = -4 g_{0,2,1,1}\left(t_6\right)-4 g_{0,2,2,0}\left(t_6\right)-\frac{1}{16} w\left(\frac{t_6}{4}\right)+\frac{1}{16} x\left(\frac{t_6}{4}\right),
    \\
    & g_{1,2,1,0}\left(t_6\right) = -g_{0,3,1,0}\left(t_6\right),
    \\
    & g_{1,3,0,0}\left(t_6\right) = -\frac{1}{2} g_{0,0,2,2}\left(t_6\right)-g_{0,4,0,0}\left(t_6\right) 
    -\frac{1}{64} w\left(\frac{t_6}{4}\right)
    +\frac{1}{192} x\left(\frac{t_6}{4}\right)^2
    -\frac{1}{96} y\left(\frac{t_6}{4}\right)^2
\end{align*}
and we are left with $5$ unknown functions 
$g_{0,0,2,2}\left(t_6\right)$, $g_{0,2,1,1}\left(t_6\right)$, $g_{0,2,2,0}\left(t_6\right)$, $g_{0,3,1,0}\left(t_6\right)$ and $g_{0,4,0,0}\left(t_6\right)$.

WDVV equation can be solved explicilty giving two solutions:
\begin{align}
    & g_{0,0,2,2}\left(t_6\right) = -\frac{1}{64} w\left(\frac{t_6}{4}\right)+\frac{1}{64} x\left(\frac{t_6}{4}\right)^2-\frac{1}{64} y\left(\frac{t_6}{4}\right)^2, \ g_{0,2,1,1}\left(t_6\right) = 0,
    \\
    & g_{0,2,2,0}\left(t_6\right) = -\frac{1}{64} w\left(\frac{t_6}{4}\right)+\frac{1}{64} x\left(\frac{t_6}{4}\right)^2-\frac{1}{64} y\left(\frac{t_6}{4}\right)^2, \ g_{0,3,1,0}\left(t_6\right) = 0,
    \\
    &g_{0,4,0,0}\left(t_6\right) = -\frac{1}{128} w\left(\frac{t_6}{4}\right)-\frac{1}{384} x\left(\frac{t_6}{4}\right)^2-\frac{1}{384} y\left(\frac{t_6}{4}\right)^2,
\end{align}
and
\begin{align}
    & g_{0,0,2,2}\left(t_6\right) = 
    -\frac{1}{64} w\left(\frac{t_6}{4}\right), \ g_{0,2,1,1}\left(t_6\right) = 0,
    \\
    & g_{0,2,2,0}\left(t_6\right) = -\frac{1}{64} w\left(\frac{t_6}{4}\right), \ g_{0,3,1,0}\left(t_6\right) = 0,
    \\
    & g_{0,4,0,0}\left(t_6\right) = -\frac{1}{128} w\left(\frac{t_6}{4}\right)+\frac{1}{192} x\left(\frac{t_6}{4}\right)^2-\frac{1}{96} y\left(\frac{t_6}{4}\right)^2.
\end{align}
It is important to note that all the functions above have Fourier series expansions in $\exp(t_6)$.

The first solution gives
\begin{align}
    \F_{\fAp,G^D} &= \frac{1}{2} t_{6} t_1^2 + \frac{1}{4} t_1 \sum_{k=1}^4 t_{k,1}^2 
    \\
    &-\frac{1}{96} \left(t_{1,1}^4+t_{1,2}^4+t_{2,1}^4+t_{2,2}^4\right)  +  t_{1,1} t_{1,2} t_{2,1} t_{2,2} e^{t_6} + O(q^2).
\end{align}
It follows that the potential obtained satisfies all conditions of Theorem~\ref{theorem: yuuki--GW} and therefore coincides with the potential $\F_{\PP^1_A}$. This proves Theorem~\ref{theorem: main} for this particular pair $(\fAp, G)$.

The second solution also gets the form as above after the linear change of the variables $t_{1,1} \to (t_{4,1}+t_{3,1})/\sqrt{2}$, $t_{1,2} \to (t_{4,1}-t_{3,1})\sqrt{2}$, $t_{2,1} \to (t_{1,1} + t_{2,1})\sqrt{2}$ and $t_{2,2} \to (t_{1,1}-t_{2,1})\sqrt{2}$.

\begin{remark}
It is important to note that if one does not impose the $\Aut$--invariance axiom, one still can solve WDVV equation in terms of unknown functions $g_\alpha(t_6)$. However in this case one two more solutions in addition to those listed above. The additional two solutions can be written explicilty in terms of functions $x(t_6)$, $y(t_6)$ and $w(t_6)$, but they define Frobenius manifolds that are not isomorphic to $M_{\PP^1_A}$. In particular, the corresponding Frobenius potentials will not be symmetric in $t_{1,1}$ and $t_{2,1}$. 
\end{remark}

We can now write the potential explicilty
\begin{align}
    \F_{\widetilde E_7,G^D}^\infty(\bv) &= -\frac{1}{128} w \left(v_{g,1}^2+v_{h,1}^2+v_{g^0,1}^2+v_{h^0,1}^2\right){}^2
    \\
    &+\frac{1}{384} x^2 \Big(6 v_{g^0,1}^2 \left(-v_{g,1}^2+v_{h,1}^2+v_{h^0,1}^2\right)+6 v_{h^0,1}^2 \left(v_{g,1}^2-v_{h,1}^2\right)
    \\
    &\quad\quad+6 v_{g,1}^2 v_{h,1}^2+v_{g,1}^4+v_{h,1}^4+v_{g^0,1}^4+v_{h^0,1}^4\Big)
    \\
    &+\frac{1}{192} y^2 \left(-6 v_{g,1}^2 v_{h^0,1}^2-v_{g,1}^4-6 v_{h,1}^2 v_{g^0,1}^2-v_{h,1}^4-v_{g^0,1}^4-v_{h^0,1}^4\right)
    \\
    &+ v_{\text{id},1} \left(\frac{v_{g,1}^2}{8}+\frac{v_{h,1}^2}{8}+\frac{v_{g^0,1}^2}{8}+\frac{v_{h^0,1}^2}{8}\right)
    +\frac{1}{2} v_{\text{id},1}^2 v_{\text{id},6}
\end{align}

\subsubsection{$A' = (4,4,2)$ with $G = K_3$}
Write $\F_{\fAp,G}(\bt)$ in the coordinates $t_{1,1} = t_{(3,1),1}$, $t_{2,1} = t_{(3,2),1}$, $t_{3,1} = t_{(3,3),1}$, $t_{4,1} = t_{(3,4),1}$.

Due to $\Aut$--invariance axiom the functions $g_\alpha$ coincide for the quadruples $\alpha$ differing by a permutation.
By using mirror theorem \ref{theorem: msUnorbifolded} we have $\F_{A'} = \F_{\PP^1_{A'}}$. The latter potential can be found in \cite{BP18,B15}. In this section keep $x = x(t)$, $y = y(t)$ and $w = w(t)$ as in Section~\ref{section: elliptic E_7 G^D case}.


Restriction of $\F_{A'}$ to the fixed locus of the group action is given by $t_{1,k} = t_{2,k} = 0$ for $1 \le k \le 3$. Set also $t_{3,1} = t_2$.

The restriction to the $\id$--sector of $\F_{\fAp,G}$ is given by setting $t_{1,1} = \dots = t_{4,1}= t_2$.
\begin{align}
    \F_{A'}^G &= \frac{1}{96} t_2^4 \left( 2 x\left(t_6\right){}^2-y\left(t_6\right){}^2- 3 w\left(t_6\right)\right)+\frac{1}{4} t_1 t_2^2+\frac{1}{2} t_6 t_1^2,
    \\
    \F_{\fAp,G}^\id &= 4 t_2^4 g_{0,0,0,4}\left(t_6\right)+12 t_2^4 g_{0,1,0,3}\left(t_6\right)+12 t_2^4 g_{0,1,1,2}\left(t_6\right)
    \\
    &+6 t_2^4 g_{0,2,0,2}\left(t_6\right)+t_2^4 g_{1,1,1,1}\left(t_6\right)+t_1 t_2^2+\frac{1}{2} t_6 t_1^2
\end{align}
Equating $\F_{\fAp,G}^\id(t_1,t_2,t_6) = 4 \cdot \F_{A'}^G (t_1,t_2,t_6/4)$ we get
\begin{align*}
    g_{1,1,1,1}\left(t_6\right) &= -4 g_{0,0,0,4}\left(t_6\right)-12 g_{0,1,0,3}\left(t_6\right)-12 g_{0,1,1,2}\left(t_6\right)-6 g_{0,2,0,2}\left(t_6\right)
    \\
    &-\frac{1}{8} w\left(\frac{t_6}{4}\right)+\frac{1}{12} x\left(\frac{t_6}{4}\right){}^2-\frac{1}{24} y\left(\frac{t_6}{4}\right){}^2
\end{align*}
and we are left with $4$ unknown functions 
$g_{0,0,0,4}\left(t_6\right)$, $g_{0,1,0,3}\left(t_6\right)$, $g_{0,1,1,2}\left(t_6\right)$, $g_{0,2,0,2}\left(t_6\right)$.

WDVV equation can be solved explicilty giving two solutions:
\begin{align}
    & g_{0,0,0,4}\left(t_6\right) =  \frac{1}{384} \left(-3 w\left(\frac{t_6}{4}\right)+2 x\left(\frac{t_6}{4}\right)^2-4 y\left(\frac{t_6}{4}\right)^2\right), \\
    & g_{0,1,0,3}\left(t_6\right) =  0, \ g_{0,1,1,2}\left(t_6\right) =  0, \ g_{0,2,0,2}\left(t_6\right) =  -\frac{1}{64} w\left(\frac{t_6}{4}\right),
\end{align}
and
\begin{align}
    & g_{0,0,0,4}\left(t_6\right) =  \frac{1}{384} \left(-3 w\left(\frac{t_6}{4}\right)-x\left(\frac{t_6}{4}\right)^2-y\left(\frac{t_6}{4}\right)^2\right),
    \\
    & g_{0,1,0,3}\left(t_6\right) =  0, \ g_{0,1,1,2}\left(t_6\right) =  0, 
    \\ 
    &g_{0,2,0,2}\left(t_6\right) =  \frac{1}{64} \left(-w\left(\frac{t_6}{4}\right)+x\left(\frac{t_6}{4}\right)^2-y\left(\frac{t_6}{4}\right)^2\right).
\end{align}
It is important to note that all the functions above have Fourier series expansions in $q = \exp(t_6)$.

The first solution gives
\begin{align}
    \F_{\fAp,G} &= \frac{1}{2} t_6 t_1^2 + \frac{1}{4} t_1 \sum_{k=1}^4 t_{k,1}^2 
    \\
    &-\frac{1}{96} \left(t_{1,1}^4+t_{1,2}^4+t_{2,1}^4+t_{2,2}^4\right)  +  t_{1,1} t_{1,2} t_{2,1} t_{2,2} e^{t_6} + O(q^2).
\end{align}
It follows that the potential obtained satisfies all conditions of Theorem~\ref{theorem: yuuki--GW} and therefore coincides with the potential $\F_{\PP^1_A}$.

The second solution also gets the form as above after the linear change of the variables $t_{1,1} \to (t_{4,1}+t_{3,1})/\sqrt{2}$, $t_{1,2} \to (t_{4,1}-t_{3,1})\sqrt{2}$, $t_{2,1} \to (t_{1,1} + t_{2,1})\sqrt{2}$ and $t_{2,2} \to (t_{1,1}-t_{2,1})\sqrt{2}$.

The potential $\F_{(\widetilde E_7,K_3)}^\infty(\bv)$ can be now be written explicitly
\begin{align*}
    \F_{(\widetilde E_7,K_3)}^\infty(\bv) &= -\frac{w}{96} \left(3 v_{g,1} v_{g^2,1}+2 v_{g^0,1}^2\right)^2
    \\
    &+\frac{x^2}{288}  \left(12 \left(v_{g^2,1}^3+v_{g,1}^3\right) v_{g^0,1}+9 v_{g,1}^2 v_{g^2,1}^2+8 v_{g^0,1}^4\right)
    \\
    &+\frac{y^2}{144}  \left(-18 v_{g,1} v_{g^2,1} v_{g^0,1}^2-6 \left(v_{g^2,1}^3+v_{g,1}^3\right) v_{g^0,1}-9 v_{g,1}^2 v_{g^2,1}^2-2 v_{g^0,1}^4\right)
    \\
    &+v_{\text{id},1} \left(\frac{1}{2} v_{g,1} v_{g^2,1}+\frac{v_{g^0,1}^2}{3}\right) +\frac{1}{2} v_{\text{id},1}^2 v_{\text{id},6}.
\end{align*}

\subsubsection{$A' = (6,3,2)$ with $G = K_3$}
Write $\F_{\fAp,G}(\bt)$ in the coordinates $t_{1,1} = t_{(1,1),1}$, $t_{2,1} = t_{(1,2),1}$, $t_{3,1} = t_{(1,3),1}$, $t_{4,1} = t_{(3,1),1}$.

Due to the symmetry axiom we have $g_{\alpha_1,\alpha_2,\alpha_3,\alpha_4} = g_{\alpha_1,\alpha_3,\alpha_2,\alpha_4} = g_{\alpha_1,\alpha_4,\alpha_3,\alpha_2}$.
By using mirror theorem \ref{theorem: msUnorbifolded} we have $\F_{A'} = \F_{\PP^1_{A'}}$. The latter potential was published in \cite{B15} (see also \cite{SZ17} for the certain correlators).

In this section for $q = \exp(t_6)$ denote
\begin{align}
    & A_3(q) := \theta_2\left(q^2\right) \theta_2\left(q^6\right)+\theta_3\left(q^2\right) \theta_3\left(q^6\right),
    \\
    & x(t) = A_3(q^6), \ y(t) = A_3(q^{12}), \ r(t) = A_3(q^3) - A_3(q^{12}),
    \\
    & w(t) = E_2(q^6).
\end{align}
The function $A_3(q^3)$ has the following Fourier series expansion
\begin{align}
    A_3(q^3) &= 1+6 q^3+6 q^9+6 q^{12}+12 q^{21}+6 q^{27}+6 q^{36}+12 q^{39}+O\left(q^{41}\right).
\end{align}
In particular, the functions $x(t)$, $y(t)$, $r(t)$ and $w(t)$ have Fourier series expansion in the powers of $q$ divisible by $3$. We have
\begin{align}
    & x(t) = 1+6 q^6+6 q^{18}+6 q^{24}+O\left(q^{40}\right), \ y(t) = 1+6 q^{12}+6 q^{36}+O\left(q^{40}\right), 
    \\ 
    &r(t) = 6 q^3+6 q^9+12 q^{21}+6 q^{27}+12 q^{39}+O\left(q^{40}\right)
    \\
    & w(t) = 1-24 q^6-72 q^{12}-96 q^{18}-168 q^{24}-144 q^{30}-288 q^{36}+O\left(q^{40}\right).
\end{align}

Restriction of $\F_{A'}$ to the fixed locus of the group action is given by setting $t_{1,1} = t_{1,2} = t_{1,4} = t_{1,5} =0$, $t_{2,1} = t_{2,2} = 0$. Set also $t_{1,3} = t_2$ and $t_{3,1} = t_3$.

Restriction to the $\id$--sector of $\F_{\fAp,G}$ is given by setting $t_{1,1} = t_{2,1} = t_{3,1} = t_2$ and $t_{4,1} = t_3$.
\begin{align}
    \F_{A'}^G &= t_3^2 t_2^2 \left(-\frac{w}{96}+\frac{x^2}{96}+\frac{x y}{48}-\frac{y^2}{48}\right) + t_2^4 \left(-\frac{w}{576}-\frac{x^2}{576}-\frac{x y}{288}+\frac{y^2}{288}\right)
    \\
    &+t_3^3 t_2 \left(\frac{r x}{36}+\frac{r y}{36}\right)+t_3^4 \left(-\frac{w}{64}+\frac{x^2}{192}+\frac{x y}{96}-\frac{y^2}{96}\right)+\frac{1}{2} t_6 t_2^2+\left(\frac{t_1^2}{12}+\frac{t_3^2}{4}\right) t_1,
    \\
    \F_{\fAp,G}^\id &= t_2^4 g_{4,0,0,0}\left(t_6\right) +3 t_3 t_2^3 g_{3,0,0,1}\left(t_6\right) + 3 t_2^2 t_3^2 \left(g_{2,0,0,2}\left(t_6\right)+g_{2,0,1,1}\left(t_6\right)\right) 
    \\
    &+t_3^3 t_2 \left(3 g_{1,0,0,3}\left(t_6\right)+6 g_{1,0,1,2}\left(t_6\right)+g_{1,1,1,1}\left(t_6\right)\right)
    \\
    &+t_3^4 \left(3 g_{0,0,0,4}\left(t_6\right)+6 g_{0,0,1,3}\left(t_6\right)+3 g_{0,0,2,2}\left(t_6\right)+3 g_{0,1,1,2}\left(t_6\right)\right)
    \\
    & +\frac{1}{2} t_6 t_1^2 + \frac{1}{4}t_1 t_2^2 +\frac{3}{4} t_1 t_3^2.
\end{align}
Equating $\F_{\fAp,G}^\id(t_1,t_2,t_3,t_6) = 3 \cdot \F_{A'}^G (t_1,t_2,t_3,t_6/3)$ we get
\begin{align*}
    g_{0,1,1,2}\left(t_6\right) &=  -g_{0,0,0,4}\left(t_6\right)-2 g_{0,0,1,3}\left(t_6\right)-g_{0,0,2,2}\left(t_6\right)
    \\
    &-\frac{1}{64} w\left(\frac{t_6}{3}\right)+\frac{1}{96} x\left(\frac{t_6}{3}\right) y\left(\frac{t_6}{3}\right)+\frac{1}{192} x\left(\frac{t_6}{3}\right)^2 -\frac{1}{96} y\left(\frac{t_6}{3}\right)^2 ,
    \\
    g_{1,1,1,1}\left(t_6\right) &=  -3 g_{1,0,0,3}\left(t_6\right)-6 g_{1,0,1,2}\left(t_6\right)
    \\
    &+\frac{1}{12} r\left(\frac{t_6}{3}\right) x\left(\frac{t_6}{3}\right)+\frac{1}{12} r\left(\frac{t_6}{3}\right) y\left(\frac{t_6}{3}\right),
    \\
    g_{2,0,1,1}\left(t_6\right) &=  -g_{2,0,0,2}\left(t_6\right)-\frac{1}{96} w\left(\frac{t_6}{3}\right)+\frac{1}{48} x\left(\frac{t_6}{3}\right) y\left(\frac{t_6}{3}\right)
    \\
    &+\frac{1}{96} x\left(\frac{t_6}{3}\right)^2 -\frac{1}{48} y\left(\frac{t_6}{3}\right)^2 ,
    \\
    g_{4,0,0,0}\left(t_6\right) &=  -\frac{1}{192} w\left(\frac{t_6}{3}\right)-\frac{1}{96} x\left(\frac{t_6}{3}\right) y\left(\frac{t_6}{3}\right)
    -\frac{1}{192} x\left(\frac{t_6}{3}\right)^2 +\frac{1}{96} y\left(\frac{t_6}{3}\right)^2 ,
    \\
    g_{3,0,0,1}\left(t_6\right) &=  0.
\end{align*}
and we are left with $5$ unknown functions 
$g_{0,0,0,4}\left(t_6\right)$, $g_{0,0,1,3}\left(t_6\right)$, $g_{0,0,2,2}\left(t_6\right)$, $g_{1,0,0,3}\left(t_6\right)$, $g_{1,0,1,2}\left(t_6\right)$ and $g_{2,0,0,2}\left(t_6\right)$.

WDVV equation can be solved explicilty giving two solutions:
\begin{align}
    & g_{0,0,0,4}\left(t_6\right) =  -\frac{11}{2592}+O\left(q^2\right), \ g_{0,0,1,3}\left(t_6\right) =  -\frac{1}{324}+O\left(q^2\right),
    \\
    & g_{0,0,2,2}\left(t_6\right) =  -\frac{1}{54}+O\left(q^2\right), \ g_{1,0,0,3}\left(t_6\right) =  -\frac{4 }{27}q +O\left(q^2\right),
    \\
    & g_{1,0,1,2}\left(t_6\right) =  \frac{2}{9}q +O\left(q^2\right), \ g_{2,0,0,2}\left(t_6\right) =  O\left(q^2\right),
\end{align}
and
\begin{align}
    & g_{0,0,0,4}\left(t_6\right) =  -\frac{1}{96}+O\left(q^2\right), \ g_{0,0,1,3}\left(t_6\right) =  O\left(q^2\right),
    \\
    &g_{0,0,2,2}\left(t_6\right) =  O\left(q^2\right), \ g_{1,0,0,3}\left(t_6\right) =  O\left(q^2\right),
    \\
    &g_{1,0,1,2}\left(t_6\right) =  O\left(q^2\right), \ g_{2,0,0,2}\left(t_6\right) =  O\left(q^2\right).
\end{align}
where $q = \exp(t_6)$.

The first solution gives
\begin{align}
    \F_{\fAp,G} &= \frac{1}{2} t_6 t_1^2 + \frac{1}{4} t_1 \sum_{k=1}^4 t_{k,1}^2 
    \\
    &-\frac{1}{96} \left(t_{1,1}^4+t_{1,2}^4+t_{2,1}^4+t_{2,2}^4\right)  +  t_{1,1} t_{1,2} t_{2,1} t_{2,2} e^{t_6} + O(q^2).
\end{align}
It follows that the potential obtained satisfies all conditions of Theorem~\ref{theorem: yuuki--GW} and therefore coincides with the potential $\F_{\PP^1_A}$.

The second solution also gets the form as above after the linear change of the variables 
$t_{1,1} \mapsto t_{1,1}$, $t_{3,1} \mapsto t_{2,1}/3+ 2 t_{4,1}/3 - 2 t_{3,1}/3$, $t_{3,2} \mapsto -2t_{2,1}/3  + t_{3,1}/3 + 2 t_{4,1}/3$, $t_{3,3} \mapsto - 2 t_{2,1}/3 -  2 t_{3,1}/3 -  t_{4,1}/3$. This means that $M_{\widetilde E_8,K_3}^\infty$ is defined unquely up to isomorphism.

After solving WDVV equation we can write down the potential $\F_{\widetilde E_8,K_3}(\bv)$ explicilty. In the current case it will be written via many quasimodular functions above. However it looks to be easier to use the mirror theorem and the respective change of coordinates in order to get $\F_{\widetilde E_8,K_3}(\bv)$ from the potential $\F_{\PP^1_{2,2,2,2}}$.

\section{Relations in the ring of quasimodular forms}\label{section: modular forms relations}

In Section~\ref{section: examples} we have proved that Theorem~\ref{theorem: main} holds for the pairs $(\widetilde E_7, G^D)$ and $(\widetilde E_7,K_3)$. We have also  
found potentials $\F_{\widetilde E_7,G^D}^\infty(\bv)$ and $\F_{\widetilde E_7,K_3}^\infty(\bv)$ explicitly. However by Theorem~\ref{theorem: yuuki--GW} it follows that in $\bt$ coordinates these potentials coincide with Frobenius manifold potential of $\PP^1_{2,2,2,2}$.

The Frobenius manifold potential of $\PP^1_{2,2,2,2}$ was given explicitly in \cite{st:3}. Consider the functions
\begin{align}
    & f_0^{\text{\text{ST}}}(t) := - \frac{1}{48} \left( E_2(q) - E_2(-q) \right), \ f_1^{\text{\text{ST}}}(t) := -\frac{1}{24} E_2(q^4)
    \\
    & f_2^{\text{\text{ST}}}(t) := -\frac{1}{24} E_2(q) - f_0(t) - f_1(t).
\end{align}
Then we have by \cite[Theorem~2.1]{st:3}:
\begin{equation}\label{eq: PP^1_{2,2,2,2} GW potential}
\begin{aligned}
    \F_{\PP^1_{2,2,2,2}} &= \frac{1}{2} t_6 t_1^2 + \frac{1}{4} t_1 \sum_{k=1}^4 t_{k,1}^2 + \frac{1}{12} \left( \sum_{i,j=1}^4 t_{i,1}^2t_{j,1}^2 \right) f_2^{\text{\text{ST}}}(t_6)
    \\
    & + \frac{1}{4} \left(t_{1,1}^4+t_{1,2}^4+t_{2,1}^4+t_{2,2}^4\right) f_1^{\text{\text{ST}}}(t_6)  +  t_{1,1} t_{2,1} t_{3,1} t_{4,1} f_0^{\text{\text{ST}}}(t_6)
\end{aligned}
\end{equation}


Writing $\F_{\widetilde E_7,G^D}^\infty(\bt)$ in the $\bt$ coordinates we have
\begin{align}
    4 \cdot \F_{\widetilde E_7,G^D}^\infty(\bt) &= \frac{1}{2} t_6 t_1^2 + \frac{1}{4} t_1 \sum_{k=1}^4 t_{k,1}^2  
    +  \frac{1}{16} y\left( \frac{t_6}{4} \right)^2 t_{1,1} t_{2,1} t_{3,1} t_{4,1}
    \\
    &+ \frac{1}{2 \cdot 64} \left( -w \left( \frac{t_6}{4} \right) + x\left( \frac{t_6}{4} \right)^2 - y\left( \frac{t_6}{4} \right)^2\right)\sum_{i,j=1}^4 t_{i,1}^2t_{j,1}^2
    \\
    & + \frac{1}{384} \left(-3 w\left( \frac{t_6}{4} \right) - x\left( \frac{t_6}{4} \right)^2 - y\left( \frac{t_6}{4} \right)^2\right) \sum_{k=1}^4 t_{k,1}^4  .
\end{align}
Mirror theorem~\ref{theorem: main} gives that $4 \cdot \F_{\widetilde E_7,G^D}^\infty(\bt) = \F_{\PP^1_{2,2,2,2}}(\bt)$, and we get
\begin{align}
    & \frac{1}{4} f_1^{\text{\text{ST}}}(q^4) = \frac{1}{384} \left(-3 w\left(q\right) - x\left(q\right)^2 - y\left(q\right)^2\right), 
    \\ 
    & f_0^{\text{\text{ST}}}(q^4) = \frac{1}{16} y\left(q\right)^2, \ f_2^{\text{\text{ST}}}(q^4) = - \frac{1}{32} w \left(q\right).
\end{align}
for $x,y,w$ as in Section~\ref{section: elliptic E_7 G^D case}. This gives us a non--trivial relation between the quasimodular forms $E_2(q)$ and $\theta_k(q)$:
\begin{align}
    & E_2(q) - E_2(-q) = \frac{1}{8} \theta_2(q^4)^4,
    \\
    & E_2\left(q^2\right) = \frac{1}{2} E_2\left(q\right)+\frac{1}{2} \theta_2\left(q^2\right)^4+\frac{1}{2}\theta_3\left(q^2\right)^4.
\end{align}
We show this phenomenon in more details below.

\subsection{Combining mirror isomorphisms}\label{section: selfdual cases}
    Let $\tau_{eq}$ and $\tau_{cl}$ stand for the mirror isomorphisms of Theorem~\ref{theorem: main} and Theorem~\ref{theorem: msUnorbifolded} respectively. 
    Denote by $\phi_{utw}$ the embedding of identity sector axiom and $\pr_G$ the projection taking the $G$--invariants. Consider the following diagramm.
    \[ 
        \begin{tikzcd}
        M_{(\fAp,G)}^\infty \arrow{r}{\tau_{eq}} &  M_{\PP^1_A}  
        \\%
        \left(M_{(\fAp,\{\id\})}^\infty \right)^G \arrow[hook]{u}{\phi_{utw}} 
        \\
        M_{(\fAp,\{\id\})}^\infty \arrow{u}{\pr_G} & M_{\PP^1_{A'}} \arrow{l}{\tau_{cl}^{-1}} 
        \arrow[dashed]{uu}{\varphi}
        \end{tikzcd}
    \]
    It does not give us a map $\varphi: M_{\PP_{A'}} \to M_{\PP^1_{A}}$, because this diagramm only allows one to define $\varphi$ on the submanifold of $M_{\PP_{A'}}$. 
    We show on the examples $(\widetilde E_7, K_2)$ and $(\widetilde E_6, K_1)$ that, being considered on this submanifolds, $\varphi$ is non--trivial. 
    
    In particular, $\varphi$ gives us \textit{different} descriptions of the same Fourier series expansions of $M_{\PP^1_A}$ correlation functions. From the point of view of Landau--Ginzburg orbifold we have two different pairs $(\fAp,G)$ and $(\fAp,\{\id\})$. However the certain submanifolds in $M_{(\fAp,G)}^\infty$ and $M_{(\fAp,\{\id\})}^\infty$ turn out to be isomorphic to the same submanifold in $M_{\PP^1_A}$. Writing the composition of these isomorphisms explicilty in coordinates we get the certain identities in the ring of (quasi)modular forms.
    
    \subsection{The pair $A' = (4,4,2)$, $G = K_2$}
    
    To simplify the notation we consider ${\fAp = x^4 + y^4 + z^2 - c^{-1} xyz}$ and $K_2 = \langle g \rangle$ with $g(x,y,z) = (-x,y,-z)$.
    
    \subsubsection{Part 1: mirror map $\tau_{eq}$}
    Let $\F_{\fAp,G}$ be the Frobenius manifold potential of the pair $(\tilde E_7,G_y)$. It is the function of
    $$
    v_{\id,0},v_{\id,-1},v_{\id,x^2},v_{\id,y},v_{\id,y^2},v_{\id,y^3} \quad \text{ and }  v_{g,1},v_{g,2},v_{g,3}.
    $$    
    Let $\F_{\PP^1_A}$ be the Frobenius manifold potential of $\PP^1_A = \PP^1_{4,4,2}$. It is the functions of
    $$
    t_{0}, t_{-1}, \ t_{1,1},t_{1,2},t_{1,3} \ t_{2,1},t_{2,2},t_{2,3}, \ t_{3,1}.
    $$
    The isomorphism $\tau_{eq}: M^\infty_{\fAp,G} \cong M_{\PP^1_A}$ of Theorem~\ref{theorem: main} is given by the formulae:
    \begin{align*}
        & t_{1,k} = v_{\id,y^k} + v_{g,k}, \ t_{2,k} = v_{\id,y^k} - v_{g,k}, \quad 1 \le k \le 3
        \\
        & t_{3,1} = v_{\id,x^2},
        \\
        & t_1 = v_{\id,0} \quad t_{-1} = 2 v_{\id,-1}
    \end{align*}
    giving $2 \F_{\fAp,G}(v(\bt)) = \F_{\PP^1_A}$.
    
    The maps $\phi_{utw}$ is given by setting $v_{g,1}=v_{g,2}=v_{g,3}=0$ and $\left(M_{(\fAp,\{\id\})}^\infty \right)^G$ is a Frobenius submanifold with the potential $\F^{\id}_{\fAp,G}$ that up to the fifth order in $v_\bullet$ reads (here we use again the explicit formulae of $\PP^1_{4,4,2}$ Frobenius manifold potential as in Section~\ref{section: elliptic E_7 G^D case})
  \begin{align*}
    \F^{\id}_{\fAp,G} & = \frac{1}{2} v_{\id,-1} v_0^2 + v_0 \left(\frac{v_{x^2}^2}{8}+\frac{v_y v_{y^3}}{4} +\frac{v_{y^2}^2}{8}\right)
    + \frac{v_y^2 v_{y^2}}{8} \left(\theta_2(q^{16})^2+\theta_3(q^{16})^2\right)
    +\frac{v_{x^2} v_y^2}{8} \theta_2(q^8)^2 
    \\
    & -\left(\frac{E_2\left(q^{32}\right)}{12} -\frac{E_2\left(q^{16}\right)}{24} +\frac{E_2\left(q^8\right)}{48} \right) 
    \left(v_{x^2}^2 v_y v_{y^3} + \frac{v_{x^2}^2 v_{y^2}^2}{2} + \frac{v_{x^2}^4+v_{y^2}^4}{4} + v_y^2 v_{y^3}^2 + v_y v_{y^2}^2 v_{y^3}\right)
    \\
    &- \theta_2(q^{16})^4 \left(\frac{v_{x^2}^4+v_{y^2}^4}{192} +\frac{v_y^2 v_{y^3}^2}{64} +\frac{v_y v_{y^3} v_{y^2}^2}{32}\right)
    + \theta_2(q^{16})^2 \theta_3(q^{16})^2 \left(\frac{v_{x^2}^2 v_{y^2}^2}{16} + \frac{v_y^2 v_{y^3}^2}{32} + \frac{v_y v_{y^3} v_{y^2}^2}{16} \right)
    \\
    &+ \theta_3(q^{16})^4 \left(\frac{v_{x^2}^2 v_y v_{y^3}}{16}  + \frac{v_{x^2}^2 v_{y^2}^2}{32} + \frac{v_{x^2}^4+v_{y^2}^4}{96} + 3 \frac{v_y^2 v_{y^3}^2}{64}+\frac{v_y v_{y^2}^2 v_{y^3}}{32} \right)
    \\
    &+ \theta_2(q^8)^2 \left(\theta_2(q^{16})^2+\theta_3(q^{16})^2\right) \frac{v_{x^2} v_y v_{y^2} v_{y^3}}{16} + \text{higher order terms}.
  \end{align*}
  
  \subsubsection{Part 2: mirror map $\tau_{cl}$}
  We have $A' = A$. The mirror map $\tau_{cl}$ of Theorem~\ref{theorem: msUnorbifolded} is given by
  \begin{align}
    & t_{1,k} = v_{x^k}, \ t_{2,k} = v_{y^k}, \quad 1 \le k \le 3,
    \\
    & t_{3,1} = v_z,
    \\
    & t_1 = v_0, \ t_{-1} = v_{-1},
  \end{align}
  giving $\F_{\fAp}(v(\bt)) = \F_{\PP^1_{A'}}$.

  The map $\pr_G$ is given by setting $v_{x} = v_{x^3} = v_{z} = 0$ and $\left(M_{\fAp}^\infty \right)^G$ is a Frobenius submanifold with the potential $\F^{G}_{\fAp}$ that up to the fifth order in $v_\bullet$ reads
  \begin{align*}
    \F^{G}_{\fAp} & = \frac{1}{2} v_{\id,-1} v_0^2 + v_0 \left(\frac{v_{x^2}^2}{8}+\frac{v_y v_{y^3}}{4} +\frac{v_{y^2}^2}{8}\right)
    + \frac{v_{x^2} v_y^2}{8} \theta_2(q^8)^2 
    +\frac{v_y^2 v_{y^2}}{8} \theta_3(q^8)^2 
    \\
    &
    - \left(\frac{E_2(q^{16})}{24} -\frac{E_2(q^8)}{48} +\frac{E_2(q^4)}{96} \right) \left( v_{x^2}^2 v_y v_{y^3} + \frac{v_{x^2}^2 v_{y^2}^2}{2}  + \frac{v_{x^2}^4+v_{y^2}^4}{4} + v_y^2 v_{y^3}^2 + v_y v_{y^2}^2 v_{y^3} \right)
    \\
    &
    - \theta_2(q^8)^4 \left( \frac{v_{x^2}^4+v_{y^2}^4}{192} + \frac{v_y^2 v_{y^3}^2}{64} + \frac{v_{x^2}^2 v_y v_{y^3}}{32}\right)
    + \theta_3(q^8)^2 \theta_2(q^8)^2 \frac{v_{x^2} v_y v_{y^2} v_{y^3}}{16}
    \\
    &
    + \theta_3(q^8)^4 \left( \frac{v_{x^2}^2 v_y v_{y^3}}{32} + \frac{v_{x^2}^2 v_{y^2}^2}{64} + \frac{v_{x^2}^4+v_{y^2}^4}{384} + \frac{v_y^2 v_{y^3}^2}{64} \right)
    + \text{higher order terms}.
  \end{align*}
  \subsubsection{Comparison}
  Due to invariant sector axiom of $M^\infty_{(\fAp,G)}$ we should have $\F^G_{\fAp} \equiv \F^{\id}_{\fAp,G}$ viewed as the formal power series. Due to the quasihomogeneity property we have, this is equivalent to comparing the coefficients in $v_\bullet$ of these potentials, viewed as Frourier series in $q$. 
  This amounts to the following equations.
  \begin{align}
    & v_y^2 v_{y^2}: \quad \theta_2(q^{16})^2 + \theta_3(q^{16})^2 = \theta_3 (q^8)^2,
    \\
    & v_{x^2}^3 v_{y^3}^2: \quad \theta_2 (q^8)^5 = 4 \theta_2 (q^8) \theta_2 (q^{16})^2 \theta_3 (q^{16})^2.
  \end{align}
  These two equalities are known as double argument formulae for theta constants. Note that these equalities follow essentially from our Mirror theorem. Checking the other monomials (26 in total) we don't get any other independent equalities in this case. 
  
    \subsection{The pair $\widetilde E_6$, $G = K_1$}

    To simplify the notation we consider $f = x^3 + y^3 + z^3 - c^{-1} xyz$ and $K_1 = \langle g \rangle$.

    \subsubsection{Part 1: mirror map $\tau_{eq}$}
    Let $\F_{\fAp,G}$ be the Frobenius manifold potential of the pair $(\tilde E_6,K_1)$. It is the function of
    $$
        v_{\id,0},v_{\id,-1},v_{\id,x},v_{\id,x^2} \quad \text{ and }  v_{g,1},v_{g,2},v_{g^2,1},v_{g^2,2}.
    $$    
    Let $\F_{\PP^1_A}$ be the Frobenius manifold potential of $\PP^1_A = \PP^1_{3,3,3}$. This potential was written explicitly in \cite{st:3}. It is the functions of
    $$
    t_{0}, t_{-1}, \ t_{1,1},t_{1,2} \ t_{2,1},t_{2,2}, \ t_{3,1},t_{3,2}.
    $$
    The isomorphism $\tau_{eq}: M^\infty_{\fAp,G} \cong M_{\PP^1_A}$ of Theorem~\ref{theorem: main} is given by the formulae:
    \begin{align*}
        & t_{p,k} = v_{\id,x^k} + \epi \left[\frac{p(k-1)}{3}\right] v_{g,k} + \epi \left[\frac{p(1-k)}{3} \right] v_{g^2,k},  \quad 1 \le k \le 2, \ 1 \le p \le 3,
        \\
        & t_1 = v_{\id,0} \quad t_{-1} = 3 v_{\id,-1}
    \end{align*}
    giving $3 \F_{\fAp,G}(v(\bt)) = \F_{\PP^1_A}$.

    The maps $\phi_{utw}$ is given by setting $v_{g,1}=v_{g,2}=v_{g^2,1}=v_{g^2,2}=0$ and $\left(M_{(\fAp,\{\id\})}^\infty \right)^G$ is a Frobenius submanifold with the potential $\F^{\id}_{f,G}$ that up to the fifth order in $v_\bullet$ reads 
  \begin{align*}
    \F^{\id}_{f,G} & = \frac{1}{2} v_0^2 v_{-1} + \frac{1}{3} v_0 v_{\id,z} v_{\id,z^2} + \Bigg( \frac{1}{18} \theta_2\left(q^{18}\right) \theta_2\left(q^{54}\right)  
    + \frac{1}{18} \theta_3\left(q^{18}\right) \theta_3\left(q^{54}\right)  
    + \frac{\eta \left(q^{27}\right)^3}{3 \eta \left(q^9\right)} \Bigg)v_{\id,x^2}^3 
    \\
    &+ \Bigg(-\frac{1}{16} E_2\left(q^{27}\right)  - \frac{1}{48} E_2\left(q^9\right)  
    + \frac{1}{18} \theta_2\left(q^{18}\right){}^2 \theta_2\left(q^{54}\right){}^2  
    + \frac{1}{18} \theta_3\left(q^{18}\right){}^2 \theta_3\left(q^{54}\right){}^2 
     \\
     &
    +\frac{1}{9} \theta_2\left(q^{18}\right) \theta_2\left(q^{54}\right) \theta_3\left(q^{18}\right) \theta_3\left(q^{54}\right) +\frac{\eta \left(q^{27}\right)^6 }{2 \eta \left(q^9\right)^2} 
    \\
    &+\frac{\eta \left(q^{27}\right)^3 \theta_2\left(q^{18}\right) \theta_2\left(q^{54}\right)}{6 \eta \left(q^9\right)}  +\frac{\eta \left(q^{27}\right)^3 \theta_3\left(q^{18}\right) \theta_3\left(q^{54}\right)}{6 \eta \left(q^9\right)}  \Bigg) v_{\id,x}^2 v_{\id,x^2}^2
    \\
    &+ \text{higher order terms}.
  \end{align*}
  
  \subsubsection{Part 2: mirror map $\tau_{cl}$}
  We have $A' = A$. The mirror map $\tau_{cl}$ of Theorem~\ref{theorem: msUnorbifolded} is given by
  \begin{align}
    & t_{1,k} = v_{x^k}, \ t_{2,k} = v_{y^k}, \ t_{3,k} = v_{z^k} \quad 1 \le k \le 3,
    \\
    & t_1 = v_0, \ t_{-1} = v_{-1},
  \end{align}
  giving $\F_{\fAp}(v(\bt)) = \F_{\PP^1_{A'}}$.

  The map $\pr_G$ is given by setting $v_{y} = v_{y^2} = v_{z} = v_{z^2} = 0$ and $\left(M_{\fAp}^\infty \right)^G$ is a Frobenius submanifold with the potential $\F^{G}_{\fAp}$ that up to the fifth order in $v_\bullet$ reads
  \begin{align*}
    \F^{G}_{\fAp} & = \frac{1}{2} v_0^2 v_{-1} + \frac{1}{3} v_0 v_{x} v_{x^2} + \left( -\frac{1}{48} E_2\left(q^9\right) -\frac{1}{144} E_2\left(q^3\right) \right)v_x^2 v_{x^2}^2
    \\
    & \left( \frac{1}{18} \theta_2\left(q^6\right) \theta_2\left(q^{18}\right) +\frac{1}{18} \theta_3\left(q^6\right) \theta_3\left(q^{18}\right)  \right)v_{x^2}^3
    \\
    &+ \text{higher order terms}.
  \end{align*}
  \subsubsection{Comparison}
  Proceeding as in section above we get after some simplification the following relations
  \begin{align}
    & \frac{1}{2} \left(E_2(q)-9 E_2\left(q^9\right)\right) + 3 \left(\theta _2\left(q^6\right) \theta _2\left(q^{18}\right)+\theta _3\left(q^6\right) \theta _3\left(q^{18}\right)\right)^2
    \\
    & +\theta _2\left(q^2\right){}^2 \theta _2\left(q^6\right){}^2+2 \theta _2\left(q^2\right) \theta _3\left(q^2\right) \theta _3\left(q^6\right) \theta _2\left(q^6\right)+\theta _3\left(q^2\right){}^2 \theta _3\left(q^6\right){}^2=0
    \\
    &-\frac{6 \eta \left(q^9\right)^3}{\eta \left(q^3\right)}+\theta _2\left(q^6\right) \left(\theta _2\left(q^2\right)-\theta _2\left(q^{18}\right)\right)+\theta _3\left(q^6\right) \left(\theta _3\left(q^2\right)-\theta _3\left(q^{18}\right)\right)=0
  \end{align}
  These two identities can be easily checked with the computer to any order in $q$. They might be known to the experts bu we didn't find them in any textbook.
  Checking the other monomials we do not get any other independent identities in this case.



\begin{thebibliography}{IR}

\linespread{0.75}
\setlength{\parskip}{0.0ex}
\small{
\bibitem[AGV]{agv:1}
	D.~Abramovich, T.~Graber,  A.~Vistoli,
	{\it Gromov--Witten theory of Deligne--Muford stacks},
	Amer. J. Math. {\bf 130} (2008), no. 5, 1337--1398.
\bibitem[B14]{B14}	
	A.~Basalaev, 
	\emph{Orbifold GW theory as the Hurwitz–Frobenius submanifold}, 
	J. Geom. Phys., 77, 30--42, (2014). 
\bibitem[B15]{B15}
    A.~Basalaev, 
    \emph{Elliptic orbifolds potentials}, 
    Personal homepage, http://basalaev.wordpress.com (2015).
\bibitem[B16]{B16}
    A.~Basalaev, 
    \emph{6-dimensional FJRW theories of the simple-elliptic singularities}, 
    arXiv preprint: 1608.08962, (2016).
\bibitem[BP18]{BP18}
	A.~Basalaev, N.~Priddis, 
	\emph{Givental–type reconstruction at a non–semisimple point}, 
	Michigan Math. J., 2018. Vol. 67. No. 2. 333--369, (2018).
\bibitem[BTW16]{BTW16}
	A.~Basalaev, A.~Takahashi, E.~Werner, 
	\emph{Orbifold Jacobian algebras for invertible polynomials}, 
	arXiv preprint: 1608.08962, (2016).
\bibitem[BTW17]{BTW17}
	A.~Basalaev, A.~Takahashi, E.~Werner, 
	\emph{Orbifold Jacobian algebras for exceptional unimodal singularities}, 
	Arnold Math. J., 3, 483--498, (2017).
\bibitem[CR]{cr:1}
	W.~Chen, Y.~Ruan,
	{\it Orbifold Gromov--Witten Theory},
	Orbifolds in mathematics and physics (Madison, WI, 2001), 25--85, Contemp. Math., 
	310, Amer. Math. Soc., Providence, RI, (2002).
\bibitem[Dub2]{du:2}
	B.~Dubrovin,
	{\it Painleve' transcendents and two-dimensional topological field theory}, 
	arXiv preprint: math/9803107, (1998).
\bibitem[ET]{et:1}
	W.~Ebeling, A.~Takahashi,
	{\it A geometric definition of Gabrielov numbers}, 
	Rev Mat Complut, 27:447–460, (2014).
\bibitem[GL]{gl:1}
    W.~Geigle, H.~Lenzing,
    {\it A class of weighted projective curves arising in representation theory of finite-dimensional algebras},
    Singularities, representation of algebras, and vector bundles (Lambrecht, 1985), 9--34, Lecture Notes in Math., 1273, Springer, Berlin, (1987). 
\bibitem[KM]{km:1}
    M.~Kontsevich, Yu.~Manin,
    {\it Gromov--Witten classes, quantum cohomology, and enumerative geometry},
    Comm. Math. Phys., 164(3),  525--562, (1994). 
\bibitem[IR]{IR}
	Y.~Ito, M.~Reid: 
	{\it The McKay correspondence for finite subgroups of ${\rm SL}(3,\CC)$},
	In: Higher-dimensional complex varieties (Trento, 1994), de Gruyter, Berlin, 1996, 221--240.
\bibitem[IST1]{ist:1}
	Y. Ishibashi, Y. Shiraishi, A. Takahashi, 
	{\it A Uniqueness Theorem for Frobenius Manifolds and Gromov--Witten Theory for Orbifold Projective Lines},
	J. Reine Angew. Math., 702, 143--71, (2015).
\bibitem[IST2]{ist:2}
	Y. Ishibashi, Y. Shiraishi, A. Takahashi, 
	{\it Primitive Forms for Affine Cusp Polynomials}, 
	Tohoku Math. J., 71(3), 437--64, (2019).
\bibitem[IV]{IV}
	K. A.~Intriligator, C.~Vafa, \emph{Landau-Ginzburg Orbifolds}, 
	Nucl. Phys., B., 339:95--120, (1990).
\bibitem[K03]{K}
	R. M.~Kaufmann, \emph{Orbifolding Frobenius Algebras}, 
	Internat. J. Math., 14(06), 573--617, (2003). 
\bibitem[MR]{mr}
	T. Milanov, Y. Ruan, 
	{\it Gromov-Witten theory of elliptic orbifold $\mathbb{P}^1$ and quasi-modular forms}, 
	arXiv preprint: 1106.2321. 
\bibitem[MS]{ms}
	T. Milanov, Y. Shen, 
	{\it Global mirror symmetry for invertible simple elliptic singularities},
	Ann. Inst. Fourier (Grenoble), Vol. 66, No. 1, 271--330, (2016).
\bibitem[R]{rossi}
	P. Rossi, 
	{\it Gromov-Witten theory of orbicurves, the space of tri-polynomials and symplectic field theory of Seifert fibrations}, 
	Math. Ann. 348, 265--287, (2010). 
\bibitem[Sai]{sa:1}
	K.~Saito,
	{\it Period Mapping Associated to a Primitive Form},
	Publ RIMS, Kyoto Univ. 19, 1231--1264, (1983).
\bibitem[SaiT]{st:2}
	K.~Saito and A.~Takahashi,
	{\it From Primitive Forms to Frobenius manifolds}, 
	Proc. Sympos. Pure Math., 78, 31--48, (2008).
\bibitem[SaT]{st:3}
	I. Satake, A. Takahashi, 
	{\it Gromov--Witten invariants for mirror orbifolds of simple elliptic singularities}, 
	Ann. Inst. Fourier, Vol. 61, No. 7, 2885--2907, (2011). 
\bibitem[Shi]{shi:1}
	Y.~Shiraishi,
	{\it On Frobenius Manifolds from Gromov--Witten Theory of Orbifold Projective Lines with $r$ orbifold points},
	Tohoku Math. J., 70(1), 17--37, (2018).
\bibitem[ShiT1]{st:1}
	Y. Shiraishi, A. Takahashi, 
	{\it On the Frobenius Manifolds for Cusp Singularities}, 
	Adv. Math., 273, 485--522, (2015).
\bibitem[S20]{S20}
	D.~Shklyarov, \emph{On Hochschild invariants of Landau--Ginzburg orbifolds},
    Adv. Theor. Math. Phys., 24(1), 189--258, (2020).
\bibitem[SZ17]{SZ17}
	Y.~Shen, J. Zhou \emph{Ramanujan identities and quasi-modularity in Gromov–Witten theory},
    Commun. Number Theory Phys., 11, 405--452, (2017).
\bibitem[T11]{T11}
	A.~Takahashi \emph{Mirror symmetry between orbifold projective lines},
    Adv. Stud. Pure Math., 66, 257--282, (2015).
\bibitem[Tu19]{Tu}
 	J.~Tu, \emph{Categorical Saito theory, II: Landau-Ginzburg orbifolds}, 
 	arXiv preprint:1910.00037, (2019).
}
\end{thebibliography}
\end{document}